\documentclass{article}
\usepackage[english]{babel}
\usepackage{amsfonts}
\usepackage{amsmath,amsthm,amssymb}
\usepackage{color}
\usepackage{authblk}
\usepackage{graphicx}
\usepackage{tikz}
\usepackage{stmaryrd}
\usepackage{mathrsfs,upgreek}
\usepackage{eucal}
\usepackage{enumerate}
\usepackage{changes}

\newtheorem{theorem}{Theorem}[section]
\newtheorem{lemma}[theorem]{Lemma}

\newtheorem{corollary}[theorem]{Corollary}
\theoremstyle{definition}

\newcommand{\op}[1]{\textrm{\upshape #1}}

\newcommand{\join}{\vee}

\newcommand{\meet}{\wedge}

\newcommand{\la}{\langle}
\newcommand{\ra}{\rangle}
\newcommand{\alg}[1]{{\textbf{\upshape #1}}}  %
\newcommand{\vv}[1]{\mathsf {#1}}

\renewcommand{\a}{\alpha}

\renewcommand{\th}{\theta}
\renewcommand{\o}{\omega}

\newcommand{\sse}{\subseteq}

\newcommand{\app}{\approx}
\newcommand{\HH}{{\mathbf H}}  
\newcommand{\II}{{\mathbf I}} 
\newcommand{\SU}{{\mathbf S}} 
\newcommand{\PP}{{\mathbf P}}   
\newcommand{\VV}{{\mathbf V}}   

\newcommand{\ib}{\item[$\bullet$]}

\newcommand{\Con}[1]{\operatorname{Con}(\alg #1)}

\newcommand{\vuc}[2]{#1_1,\dots,#1_{#2}}

\newcommand{\imp}{\rightarrow}
\newcommand{\nneg}{\mathop{\sim}}

\makeatletter
\def\square{\RIfM@\bgroup\else$\bgroup\aftergroup$\fi
  \vcenter{\hrule\hbox{\vrule\@height.6em\kern.6em\vrule}\hrule}\egroup}
\makeatother
\begin{document}
\title{Varieties of K-lattices\footnote{This work was supported by the Italian {\em National Group for Algebra and Geometric Structures} (GNSAGA--INDAM).}}
\author{Paolo Aglian\`{o}\\
DIISM\\
Universit\`a di Siena\\
Italy\\
agliano@live.com
\and
Miguel Andr\'es Marcos\\
Facultad de Ingenier\'ia Qu\'imica\\ CONICET - Universidad Nacional del Litoral \\
Argentina\\
mmarcos@santafe-conicet.gov.ar
}
\date{}
\maketitle

\begin{abstract} In this paper we deal with varieties of commutative residuated lattices that arise from a specific kind of construction: the {\em twist-product} of a lattice. Twist-products
were first considered by Kalman in 1958 to deal with order involutions on plain lattices, but the extension of this concept to residuated lattices has attracted some attention lately. Here we deal mainly with varieties of such lattices, that can be obtained by applying a specific twist-product construction to varieties of integral and commutative residuated lattices.
\end{abstract}

\section*{Introduction}

A Kalman lattice, or K-lattice (introduced in \cite{BusanicheCignoli2014}) is a commutative residuated lattice coming from a specific kind of construction, i.e. the {\em twist-product} of a lattice.  A twist-product of a lattice $\alg L$, is the cartesian product of $\alg L$ with
its order dual $\alg L^\partial$  where the ordering is determined by the natural  involution $\nneg(x, y)= (y, x)$; in other words for $(x,y),(z,w) \in L \times L^\partial$, $(x,y) \le (z,w)$ if and only if $x \le z$ and $y\ge w$.  The idea of considering this construction  goes back to Kalman's paper \cite{Kalman1958}, where only pure lattices were considered. The extension of this concept to residuated lattices is due to Tsinakis and Wille \cite{TsinakisWille2006}; they considered the twist-product of a residuated lattice $\alg L$ having a greatest element $\top$  such that the element $(\top,1)$ ($1$ the monoid identity) is the dualizing element relative to the natural involution. In other words for all $(x,y) \in \alg L \times \alg L^\partial$, $\nneg(x,y) = (x,y) \imp (\top,1)$ and so $((x,y) \imp (\top,1)) \imp (\top,1) = (x,y)$.
 A K-lattice is a residuated lattice that is a subalgebra of the algebra obtained by applying the Tsinakis-Wille construction to an {\em integral} commutative residuated lattice; in this case $\top=1$ and the dualizing pair is $(1,1)$.

It is quite obvious that a K-lattice is a commutative residuated lattice; moreover K-lattices form variety $\mathsf{KL}$ \cite{BusanicheCignoli2014}. Here we explore the structure of the lattice of subvarieties of $\mathsf{KL}$, that is a rather complex object. In order to do so we characterize its atoms and its finitely generated almost minimal varieties and we describe in more details several interesting sublattices.

The paper is organized as follows. In Section \ref{prelim}, we enumerate all general results needed to tackle the problem of describing the lattice of subvarieties of $\mathsf{KL}$, while in Section \ref{Klattices} we introduce K-lattices and prove some basic algebraic properties.
Section \ref{admissub} deals with some general ways of obtaining K-lattices with the same \textsl{negative cone}.
In Sections \ref{section3} and \ref{almostmin} we describe some of the atoms of the lattice of subvarieties of $\mathsf{KL}$, and a way to obtain covers of the atoms in terms of \textsl{tight algebras}.
Finally in Section \ref{latticeofsub} we consider some special subvarieties of $\mathsf{KL}$, and study the lattice of subvarieties for those cases.

\section{Preliminaries}\label{prelim}
The first ingredient is a classical result by B. J\'onsson.

\begin{lemma}\label{jonsson} (J\'onsson's Lemma) Let $\vv K$ be a class of algebras such that $\VV(\vv K)$ is congruence distributive; then
\begin{enumerate}
\item if $\alg A$ is a finitely subdirectly irreducible algebra in $\VV(\vv K)$, then $\alg A \in \HH\SU\PP_u(\vv K)$;
\item if $\alg A,\alg B$ are finite subdirectly irreducible algebras in $\VV(\vv K)$ then $\VV(\alg A) = \VV(\alg B)$ if and only if $\alg A\cong\alg B$.
\end{enumerate}
\end{lemma}
In particular if $\vv K$ is a finite class of finite algebras and $\VV(\vv K)$ is congruence distributive, then all the finitely subdirectly irreducible algebras in $\VV(\vv K)$ are in $\HH\SU(\vv K)$.

A {\bf commutative residuated lattice} is a structure $\la A,\join,\meet,\cdot,\imp,1\ra$ such that
\begin{enumerate}
\item $\la A,\join,\meet\ra$ is a lattice;
\item $\la A,\cdot,1\ra$ is a commutative monoid;
\item $(\cdot,\imp)$ form a residuated pair w.r.t. the ordering, i.e. for all $a,b,c \in A$
$$
ab \le c\qquad\text{if and only if}\qquad a \le b \imp c.
$$
\end{enumerate}
We denote this variety by $\mathsf{CRL}$; if $1$ is the largest element in the ordering the lattice is said to be {\bf integral}. Commutative and integral residuated lattices form a variety which we call $\mathsf{CIRL}$. Note that algebras in $\mathsf{CRL}$ are congruence distributive, since they have a lattice reduct; hence J\'onsson's Lemma applies.

Both $\mathsf{CIRL}$ and $\mathsf{CRL}$ are {\em ideal determined varieties} \cite{AglianoUrsini1992}; this means that in both cases their congruences are totally determined by the 1-classes. The 1-classes are called {\bf filters} in the case of $\mathsf{CIRL}$ and {\bf convex subalgebras} in the case of $\mathsf{CRL}$; of course they form algebraic lattices, isomorphic to the congruence lattice of the algebra. It is a straightforward consequences of the general results in \cite{AglianoUrsini1992} that the two maps giving the isomorphism between $\Con A$ and the lattice of filters (or convex subalgebras) of $\alg A$ are ($\th \in \Con A$, $X$ a filter or a convex subalgebra)
$$
\th \longmapsto 1/\th \qquad  \qquad  X \longmapsto \th_X= \{(a,b): (a \imp b) \meet 1, (b \imp a) \meet 1 \in X\}.
$$
In the rest of this section we will assume that any variety we consider is a subvariety of $\mathsf{CIRL}$. There are two equations that that bear interesting consequences, i.e.
\begin{align}
&(x \imp y) \join (y \imp x) \app 1.\tag{P}\\
&x(x \imp y) \app y(y \imp x); \tag{D}
\end{align}
It can be shown (see \cite{BlountTsinakis2003} and \cite{JipsenTsinakis2002}) that a subvariety of $\mathsf{CIRL}$ satisfies (P) if and only if any algebra therein is a subdirect product of totally ordered algebras (and this implies via Birkhoff's Theorem that all the subdirectly irreducible algebras are totally ordered). Such varieties
are called {\bf representable} and the subvariety axiomatized by (P) is the largest subvariety of $\mathsf{CIRL}$ that is representable.

If a variety satisfies (D) then the lattice ordering becomes the inverse divisibility ordering: for any algebra $\alg A$ therein and for all $a,b \in A$
$$
a \le b \qquad\text{if and only if}\qquad \text{there is $ c \in A$ with $a =bc$}.
$$
Moreover it can be easily shown that $a(a \imp b) = a\meet b$. Algebras satisfying (D) are called {\bf hoops} and  they form a subvariety denote by $\mathsf{H}$. If an algebra in $\mathsf{CIRL}$ satisfies both (P) and (D) then it is a {\bf basic hoop} and the corresponding variety  is denoted by $\mathsf{BH}$.
\bigskip

\noindent{\bf Remark.} We add some words of caution here; a hoop in the sense of \cite{BlokFerr2000} is a divisible integral and commutative residuated {\em semilattice}. We have chosen the same name for its lattice counterpart for lack of a better word and for two more reasons:
\begin{enumerate}
\ib congruences on a commutative integral residuated lattice are exactly the congruences of its residuated semilattice reduct (this is clear form the description of the congruences above), so we are free to use most of the results in \cite{BlokFerr2000} for hoops in our sense;
\ib for representable varieties the distinction disappears; in fact a representable residuated semilattice is a representable residuated lattice where
$$
a \join b = ((a \imp b) \imp b) \meet ((b \imp a) \imp a).
$$
For an extended discussion, even of the noncommutative case, we direct the reader to \cite{Agliano2018c}.
\end{enumerate}
\bigskip

A powerful  tool for investigating commutative and integral residuated lattices in general  is the {\bf ordinal sum}; if $\alg A_0,\alg A_1 \in \mathsf{CIRL}$ we  put a structure on the set  $A_0\setminus \{1\} \cup A_1\setminus\{1\} \cup \{1\}$. The ordering is given by
$$
a \le b \quad\text{if and only if} \quad \left\{
             \begin{array}{l}
               \hbox{$b=1$, or} \\
               \hbox{$a \in A_0\setminus\{1\}$ and $b \in A_1\setminus\{1\}$ or} \\
               \hbox{$a,b \in A_i\setminus\{1\}$ and $a \le_{A_i} b$, $i=0,1$.}
             \end{array}
           \right.
$$
and we define
\begin{align*}
&a \imp b = \left\{
              \begin{array}{ll}
                b, &\hbox{if $a=1$;} \\
                1, &\hbox{if $b=1$;} \\
                a \imp_{A_i} b, &\hbox{if $a,b \in A_i\setminus\{1\}$ and $a \le_{A_i} b$, $i=0,1$.}
              \end{array}
            \right.\\
&a\cdot b = \left\{
              \begin{array}{ll}
                a, & \hbox{if $a \in A_0\setminus\{1\}$ and $b \in A_1$;} \\
                b, & \hbox{if $a \in A_1$ and $b\in A_0\setminus\{1\}$}\\
                a  \cdot_{A_i} b, &\hbox{if $a,b \in A_i\setminus\{1\}$ and $a \le_{A_i} b$, $i=0,1$.}
              \end{array}
            \right.
\end{align*}
If we call $\alg A_0 \oplus \alg A_1$ the resulting structure, then it is easily checked that $\alg A_0 \oplus \alg A_1$ is a semilattice ordered integral and commutative residuated monoid (and so the ordinal sum of two hoops in the sense of \cite{BlokFerr2000} always exists). It might not be a residuated lattice though and  the reason is that if $1_{A_0}$ is not join irreducible and $\alg A_1$ is not bounded we run into trouble. In fact if $a,b \in A_0\setminus\{1\}$ and $a \join_{A_0} b =1_{A_0}$ then  the upper bounds of $\{a,b\}$ all lie in $A_1$; and since $A_1$ is not bounded there can be no least upper bound of $\{a,b\}$ in $\alg A_0 \oplus \alg A_1$ and the ordering cannot be a lattice ordering. However this is the only case we have to worry about; if $1_{A_0}$ is join irreducible, then the problem disappears, and if $1_{A_0}$ is not join irreducible but $\alg A_1$ is bounded, say by $u$, then we can define
$$
a \join b = \left\{
              \begin{array}{ll}
                a, & \hbox{$a \in A_1$ and $b \in A_0$;} \\
                b, & \hbox{$a \in A_0$ and $b \in A_1$;} \\
                a \join_{A_1} b, & \hbox{if $a,b \in A_1$;} \\
                a \join_{A_0} b, & \hbox{if $a,b \in A_0$ and $a \join_{A_0} b < 1$;}\\
                u, & \hbox{if $a,b \in A_0$ and $a \join_{A_0} b = 1$;}\\
              \end{array}
            \right.
$$
We will call $\alg A_0 \oplus \alg A_1$ the {\bf ordinal sum} and we will say that the ordinal sum {\bf exists} if $\alg A_0 \oplus \alg A_1 \in \mathsf{CIRL}$. We would like to point out some facts:

\begin{enumerate}
\ib every time we deal with a class of $\mathsf{CIRL}$ for which we know somehow that the ordinal sum always exists, then we can define the ordinal sum of a (possibly) infinite family of algebras in that class; in that case the family is indexed by a totally ordered set $\la I,\le\ra$ that may or may not have a minimum;
\ib ordinal sums always exist in  the class of {\em finite} algebras in
$\mathsf{CIRL}$ and the class of {\em totally ordered} algebras in $\mathsf{CIRL}$;
\ib an algebra in $\mathsf{CIRL}$ is {\bf sum irreducible} if it is non trivial and cannot be written as the ordinal sum of at least two nontrivial algebras in $\mathsf{CIRL}$;
\ib every algebra in $\mathsf{CIRL}$ is the ordinal sum of sum irreducible algebras in $\mathsf{CIRL}$ (by a straightforward application of Zorn's lemma, see for instance Theorem 3.2 in \cite{Agliano2018b});
\ib in general we have no idea of what the sum irreducible algebras in a subvariety of $\mathsf{CIRL}$ may be, the best result is the classification of all totally ordered sum irreducible basic hoops \cite{AglianoMontagna2003}.
\end{enumerate}
In general a proper subvariety of $\mathsf{CIRL}$  is not closed under ordinal sums (it is very easy to construct an ordinal sum of two  prelinear algebras in $\mathsf{CIRL}$ that is not prelinear). We do not know which equations are preserved by ordinal sums; however
\begin{enumerate}
\ib any join free equation in one variable is preserved;
\ib equation (D) is preserved;
\ib equation (P) is not preserved.
\end{enumerate}
In particular if the ordinal sum of two divisible and idempotent algebras in $\mathsf{CIRL}$ exists then it is again divisible and idempotent.

Since totally ordered algebras in $\mathsf{CIRL}$, which we may call {\bf chains} are always summable it is worth introducing some classes of them. First note that there is only one 2-element algebra in $\mathsf{CIRL}$, it is of course a chain and we will denote it by $\mathbf 2$. Idempotent hoops form the  variety $\mathsf{Br}$ of {\bf Brouwerian lattices}; of course there the meet and the product operations coincide. The prelinear subvariety of $\mathsf{Br}$ is the variety $\mathsf{GH}$ of {\bf G\"odel hoops} \cite{AFM};  a {\bf G\"odel chain} of $n$ elements, $\alg G_n$ from now on, is therefore a totally ordered Brouwerian lattice. It can be shown that $\mathsf{GH}$ is a locally finite variety, it is generated by all the finite G\"odel chains and that $\alg G_n = \mathbf 2 \oplus \dots \oplus \mathbf 2$ ($n-1$ summands).

The subvariety of hoops satisfying {\em Tanaka's equation}
\begin{equation}
(x \imp y) \imp y \app (y \imp x) \imp x \tag{T}.
\end{equation}
is the variety $\mathsf{WH}$ of {\bf Wajsberg hoops} and it is easily seen that $\mathsf{WH}$ is a prelinear variety. Varieties of Wajsberg hoops have been analyzed thoroughly in \cite{AglianoPanti1999}. Each variety is finitely axiomatizable and the axioms can be found in an algorithmic (albeit very complex) way. The $n+1$-element {\bf Wajsberg chain} $\alg \L_n$  has universe $0=a^{n} < a^n < a^{n-1} < \dots < a < a^0=1$ and the operations are defined as
$$
a^ra^s = a^{\min\{r+s,n\}}\qquad\qquad a^r \imp a^s = a^{\max \{s-r,0\}}.
$$
It is obvious that each $\alg \L_n$ is simple and $\alg \L_n \le \alg \L_m$ if and only if $n\mathrel{|}m$; moreover it can be shown that $\mathsf{WH}$ is generated by all the finite Wajsberg chains.

If $\alg A \in \mathsf{CRL}$ an element $a \in A$ is {\bf involutive} if for all $b \in A$, $(b \imp a) \imp a = b$. An algebra $\alg A \in \mathsf{CRL}$ is {\bf 1-involutive} if $1$ is an involutive element of $\alg A$. An algebra $\alg A \in \mathsf{CRL}$ is {\bf 0-involutive} or just {\bf involutive} if it is bounded (say by $0$) and $0$ is an involutive element of $\alg A$. It is easily seen that any bounded Wajsberg hoop is involutive and that a bounded hoop is involutive if and only if it is a Wajsberg hoop.

\section{K-lattices}\label{Klattices}

If $\alg A \in \mathsf{CRL}$ we denote by $A^-$ the set $\{a \in A: a \le 1\}$; it is well known that $A^-$ can be given a structure of an algebra in $\mathsf{CIRL}$.
More precisely $\alg A^- = \la A^-,\join,\meet,\cdot,\imp_1,1\ra \in \mathsf{CIRL}$, where  $a \imp_1 b := (a \imp b) \meet 1$.

\begin{lemma} \label{congruences}\cite{HartRafterTsinakis2002}  For any $\alg A \in \mathsf{CRL}$ the following four lattices are isomorphic:
\begin{enumerate}
\item $\Con A$;
\item the lattice of convex subalgebras of $\alg A$;
\item $\op{Con}(\alg A^-)$;
\item the lattice of filters of $\alg A$.
\end{enumerate}
In particular for each $\a \in \Con A$ there is an $\a^- \in \op{Con}(\alg A^-)$ such that $(\alg A/\a)^- \cong \alg A^-/\a^-$.
\end{lemma}

\begin{lemma}\label{subdir}  Let $\alg A,\alg B  \in \mathsf{CRL}$ and let $(\alg A _i)_{i\in I}) \sse \mathsf{CRL}$. Then
\begin{enumerate}
\item if $\alg A \le \alg B$, then $\alg A^- \le \alg B^-$;
\item $(\Pi_{i \in I} \alg A_i)^- \cong \Pi_{i \in I} \alg A^-$.
\end{enumerate}
\end{lemma}

Let $\alg A \in \mathsf{CIRL}$; the {\bf K-expansion} $K(\alg A)$ of $\alg A$ is a structure whose universe is $A\times A$ and the operations are defined as: 
\begin{align*}
&\la a,b\ra \join \la c,d\ra := \la a \join c, b \meet d\ra \\
&\la a,b\ra \meet \la c,d\ra := \la a \meet  c, b \join d\ra\\
&\la a,b\ra \la c,d\ra := \la ac,( a \imp d) \meet (c \imp b)\ra\\
&\la a,b\ra \imp \la c,d\ra := \la (a \imp c) \meet (d \imp b),ad\ra.
\end{align*}

\begin{theorem} For every $\alg A \in \mathsf{CIRL}$, $K(\alg A)$ is a commutative residuated lattice that is also
\begin{enumerate}
\item  $1$-involutive;
\item {\bf $1$-distributive}, i.e. it satisfies both distributive laws for lattices whenever at least one of the elements is equal to $1$.
\item If we set $\nneg x = x \imp 1$ then it satisfies the equations 
\begin{align*}
& xy \meet 1 \app (x\meet 1)(y\meet 1) \tag{K1}\\
&((x \meet 1) \imp y) \meet ((\nneg y \meet 1) \imp \nneg x) \app x \imp y.  \tag{K2}
\end{align*}
\end{enumerate}
\end{theorem}

\begin{lemma}\label{Kproperties} Let $\alg A, \alg B \in \mathsf{CIRL}$ and let $(\alg A_i)_{i\in I} \sse \mathsf{CIRL}$. Then
\begin{enumerate}
\item if $f:\alg A \longrightarrow \alg B$ is a homomorphism, then $f^K : K(\alg A) \longrightarrow K(\alg B)$ defined by $f^K(a,b) = (f(a),f(b))$ is a homomorphism;
\item if $\alg A \le \alg B$, then $K(\alg A) \le K(\alg B)$;
\item $K(\Pi_{i \in I} \alg A_i) \cong \Pi_{i \in I}K(\alg A_i)$.
\end{enumerate}
Hence for any subclass $\vv K$ of $\mathsf{CIRL}$, $\VV(K(\vv K)) = K(\VV(\vv K))$.
\end{lemma}
\begin{proof} Only point 3. needs a nontrivial argument.   We will write $\mathbf a$ for $(a_i)_{i \in I} \in \Pi_{i \in I} A_i$; let's define a map $g:K(\Pi_{i \in I} \alg A_i) \longmapsto \Pi_{i \in I}K(\alg A_i)$ by setting
$$
g(\la\mathbf a,\mathbf b\ra) = (\la a_i,b_i\ra)_{i \in I}.
$$
The map is clearly injective and surjective, so we only have to show that it respects the operations; we will work out the case for $\cdot$ and leave the rest to the reader. Let's compute
\begin{align*}
\la \mathbf  a, \mathbf b\ra\cdot \la \mathbf c, \mathbf d\ra &= \la \mathbf a \mathbf c , (\mathbf a \imp \mathbf d) \meet \mathbf c \imp \mathbf b)\ra \\
&= \la (a_ic_i)_{i \in I}, ((a_i \imp d_i) \meet (c_i \imp d_i))_{i \in I}\ra.
\end{align*}
So
\begin{align*}
g(\la \mathbf  a, \mathbf b\ra\cdot \la \mathbf c, \mathbf d\ra)&= (\la a_ic_i, (a_i \imp d_i) \meet (c_i \imp d_i)\ra)_{i \in I}\\
 &= (\la a_i,b_i\ra \cdot \la c_i,d_i\ra)_{i \in I} = (\la a_i,b_i\ra )_{i \in I} \cdot (\la c_i,d_i\ra)_{i \in I}\\
 &= g(\la \mathbf a, \mathbf b\ra) \cdot g(\la \mathbf c,\mathbf d \ra).
 \end{align*}
\end{proof}

A {\bf Kalman lattice} or just {\bf K-lattice} is a a commutative integral residuated lattice that is $1$-involutive, $1$-distributive and satisfies (K1) and (K2); the variety of K-lattices is denoted by $\mathsf{KL}$.

An algebra $\alg A$ has the {\bf congruence extension property} (CEP) if for any subalgebra $\alg B \le \alg A$, $\Con B= \{\th \cap B^2: \th \in \Con A\}$. In particular if $\alg A$ has the CEP and it is simple, then any subalgebra of $\alg A$ is simple. Moreover if $\alg A$ has the CEP, then $\SU\HH(\alg A) = \HH\SU(\alg A)$ (see \cite{BurrisSanka} Ch.II, \S 9, Exercise 5). It is well-known that any algebra in $\mathsf{CIRL}$ has the CEP and so, by Lemma \ref{congruences}, any algebra in $\mathsf{KL}$ has the CEP. If $\vv K$ is a class  of algebras in $\mathsf{CIRL}$, then $K(\vv K) = \{\alg A\in \mathsf{KL}: \alg A \le  K(\alg L)\ \text{for some $\alg L \in \vv V$}\}$; if $\vv W$ is a class of algebras in $\mathsf{KL}$ we define $\vv W^- =\VV(\{\alg A^-: \alg A \in \vv W\})$.

\begin{lemma}\cite{BusanicheCignoli2014}\label{inclusions} If $\alg L \in \mathsf{CIRL}$ then $\alg L \cong K(\alg L)^-$. If $\alg A \in \mathsf{KL}$, then $f:a \longmapsto (a \meet 1, \nneg a \meet 1)$ is an embedding of $\alg A $ in $K(\alg A^-)$.
\end{lemma}

From now on if $\vv V$ is any subvariety of $\mathsf{CRL}$ we will denote by $\Lambda(\vv V)$ its lattice of subvarieties; this lattice is complete, dually algebraic and (since $\vv V$ is congruence distributive) it is also completely distributive. We will show that $K$ and $^-$ can be seen as operators between the two  lattices $\Lambda(\mathsf{CIRL})$ and $\Lambda(\mathsf{KL})$. A subvariety $\vv W$ of $\mathsf{KL}$ is a {\bf Kalman variety} if $\vv W = K(\vv W^-)$; the following lemma connects some of the definitions we have given so far.

\begin{lemma}\label{techlemma}Let $\vv V$ be any subvariety of $\mathsf{CIRL}$ and $\vv W$ any subvariety of $\mathsf{KL}$ and $\vv K$ any subclass of $\mathsf{CRL}$:
\begin{enumerate}
 \item $K(\vv V)$ is a subvariety of $\mathsf{KL}$ and $K(\vv V) = \{\alg A\in \mathsf{KL}: \alg A^- \in \vv V\}$;
 \item  $K(\VV(\vv K)) = \VV(K(\vv K))$;
\item $\HH\SU\PP_u(K(\vv K)) \sse \SU K(\HH\SU\PP_u(\vv K))$;
 \item  $K(\vv V)^-= \vv V$ and $\vv W \sse K(\vv W^-)$;
\item  $K(\vv W^-)$ is the smallest Kalman variety containing $\vv W$;
\item $K:\Lambda(\mathsf{CIRL})  \longmapsto\Lambda(\mathsf{KL})$ is a lattice homomorphism;
\item  $\vv W^- \sse \vv V$ if and only if $\vv W \sse K(\vv V)$,
 hence\footnote{i.e. the operators $K,^-$ form a {\em left adjoint pair} between $\Lambda(\mathsf{CIRL})$ and $\Lambda(\mathsf{KL})$.}
$$
\vv W^- = \bigwedge\{\vv U: \vv W \sse K(\vv U)\}.
$$
\item $K$ is also injective, i.e. it is an embedding.
\end{enumerate}
\end{lemma}
\begin{proof} For 1. $K(\vv V)$ is closed under subalgebras by definition; it is also closed under direct products by Lemma \ref{Kproperties}. So assume that $\alg A \in  K(\vv V)$ and let $\th \in \Con A$; $\alg A \le K(\alg L)$ for some $\alg L \in \vv V$, so $\alg A^-\le \alg L$. By Lemma \ref{congruences} there is a $\th^- \in \op{Con}(\alg A^-)$ with $(\alg A/\th)^- = \alg A^-/\th^-$ and, since $\mathsf{CIRL}$ has the CEP, there is a $\psi \in \Con L$ with  $\alg A^-/\th^- \le \alg L/\psi \in \vv V$.  So $\alg A/\th \le K(\alg A^-/\th^-) \le K(\alg L/\psi)$ and thus $\alg A/\th \in  K(\vv V)$.

For 2. we observe that $K(\VV(\vv K))$ is a variety by 1. and obviously contains $K(\vv K)$. Suppose now that $K(\vv K) \sse \vv W$ where $\vv W$ is any variety; if $\alg A \in K(\VV(\vv K))$ then there is an algebra $\alg L \in \VV(\vv K)$ with $\alg A \le K(\alg L)$. By the CEP, there is a family $(\alg L_i)_{i \in I} \sse \vv K$ and a congruence $\th$ of $\Pi_{i \in I}\alg L_i$ such that $\alg L \le \Pi_{i \in I} \alg L_i /\th$.
It follows that $\alg A \le K(\alg L) \le \Pi_{i \in I} K(\alg L_i)/\psi$ (using the CEP again). Since $K(\vv K) \sse \vv W$ this implies that $\alg A \in \vv W$ and hence $K(\VV(\vv K)) \sse \vv W$. This proves the claim.

Suppose now that $\alg A \in P_u(K(\vv K))$; then there is a family $(\alg A_i)_{i \in I}$ of algebras in $K(\vv K)$ and an ultrafilter $U$ on $I$ such that $\alg A \cong \Pi_{i\in I} \alg A_i/\sim_U$ where
$(\mathbf a,\mathbf b) \in \sim_U$ if $\{i\in I: a_i=b_i\} \in U$. By Lemmas \ref{subdir} and \ref{congruences} $(\Pi_{\in \in I} \alg A_i/\sim_U)^- = \Pi_{i \in I} \alg A^-/\sim_U^-$; moreover $\alg A_i \le K(\alg L_i)$ with $\alg L_i \in \vv K$ and thus $K(\alg A_i^-) \le K(\alg L_i)$. Putting everything together we get
$$
\Pi_{i \in I} \alg A_i/\sim_U \le K(\Pi_{i \in I} \alg A^-/\sim_U^-) = \Pi_{i \in I} K(\alg A_i)/\sim_U \le \Pi_{i\in I} K(\alg L_i)/\sim_U.
$$
Thus we have shown that $\PP_u(\vv K) \sse \SU K(P_u(\vv K))$ from which 3. follows.

Suppose $\alg A \le K(\alg L)$ for some $\alg L \in \vv V$; then
$\alg A^- \le K(\alg L)^- \cong \alg L \in \vv V$ (by Lemma \ref{inclusions}). Conversely if $\alg A^- \in \vv V$, then $\alg A \le K(\alg A^-) \in \vv W$ again by Lemma \ref{inclusions}. So  4. holds.  For 5. suppose that $\vv W \sse K(\vv V)$ for some $\vv V \sse \mathsf{CIRL}$; then
for all $\alg A \in \vv W$, $\alg A \le K(\alg B)$ for some $\alg B \in \vv V$. Then $\alg A^- \le \alg B$, so $\vv W^- \sse \vv V$ and hence $K(\vv W^-) \sse K(\vv V)$.

Next, using Lemma \ref{Kproperties} we compute
\begin{align*}
\alg A &\in  K(\vv V \join \vv W) =K(\VV(\vv V \cup \vv W))\\
& \VV(K(\vv V \cup \vv W)) = \VV(K(\vv V) \cup K(\vv W)) = K(\vv V) \join K(\vv W);
\end{align*}
The other implication follows from the fact that $K$  respects the ordering; that $K$ respects also the meet is trivial and 6. holds.

For 7. suppose that $\vv W^- \sse \vv V$; then by 3. $\vv W \sse K(\vv W^-) \sse K(\vv V)$. Conversely if $\vv W \sse K(\vv V)$, then $\vv W^- \sse K(\vv V)^- = \vv V$.
Hence if $\vv V ,\vv U \in \Lambda(\mathsf{CIRL})$,  $K(\vv V) \sse K(\vv U)$ if and only if $\vv V = K(\vv V)^- \sse \vv U$; this proves 8.
\end{proof}

 By Lemma \ref{inclusions} any subvariety of $\mathsf{KL}$ of the form $K(\vv V)$ for some subvariety $\vv V$ of $\mathsf{CIRL}$ is a Kalman variety. Moreover we know that $\mathsf{KL}$ is Kalman since $\mathsf{KL} = K(\mathsf{CIRL})$ \cite{BusanicheCignoli2014}. Another interesting Kalman variety is the variety $\mathsf{NPcL}$ of {\bf Nelson paraconsistent lattices} (that are the equivalent algebraic semantics of a conservative extension of Nelson's paraconsistent logic N4); the variety $\mathsf{NPcL}$ consist of K-lattices that are distributive and also satisfy
 $(x \meet 1)^2 \app x \meet 1$. In \cite{BusanicheCignoli2009} the authors proved that $\mathsf{NPcL} = K(\mathsf{Br})$. Their argument depends on an interesting fact, whose proof is more or less routine, but that is convenient to spell out.

\begin{lemma}\label{eqn} Any variety of commutative residuated lattices satisfying (K1) satisfies also
$$
(x \meet 1) \imp_1 y \app (x \meet 1) \imp_1 (y \meet 1).
$$
\end{lemma}

Thanks to this lemma se can write down a procedure to axiomatize $K(\vv V)$, modulo an axiomatization of $\vv V$. First observe that any equation in the language of $\mathsf{CIRL}$ is equivalent to an equation $p(\vuc xn) \app 1$ for some $p \in \alg F_{\mathsf{CIRL}}(\o)$. Next we define a map $\kappa: p \in \alg F_{\mathsf{CRL}}(\o) \mapsto p \in \alg F_{\mathsf{CRL}}(\o)$ by induction on the formation of terms
$$
\kappa(p) =\left\{
             \begin{array}{ll}
               x\meet 1, & \hbox{if $p=x$ or $p=1$;} \\
               \kappa(r \meet 1) \join \kappa(s\meet 1), & \hbox{if $p = r \join s$;} \\
               \kappa(r\meet 1) \meet \kappa(s\meet 1), & \hbox{if $p = r \meet s$;} \\
               \kappa(r\meet 1) \kappa(s\meet 1), & \hbox{if $p = r s$;} \\
               \kappa(r\meet 1) \imp_1 \kappa(s), & \hbox{if $p = r \imp s$.} \\
             \end{array}
           \right.
$$

\begin{theorem}\label{axioms} If $\vv V$ is a subvariety of $\mathsf{CIRL}$ axiomatized by $\{ p_i \app 1: \in I\}$ then $K(\vv V)$ is axiomatized by $\{\kappa(p_i) \app 1: i \in \}$.
\end{theorem}

As an example (introduced in \cite{ABGM2017}) we consider a subvariety of $\mathsf{NPcL}$,
 more precisely the subvariety $\mathsf{GNPcL}$ axiomatized by the single equation
\begin{equation}
(((x \meet 1) \imp y) \join ((y \meet 1) \imp x)) \meet 1 \app 1.
\end{equation}
Since $\mathsf{KL}$ is $1$-distributive  we can rewrite the above equation as
$$
(( x \meet 1) \imp_1 y) \join ((y \meet 1) \imp_1 x) \app 1
$$
and so it is clear that for any $\alg A \in \mathsf{GNPcL}$, $\alg A^-$ satisfies
$$
(x \imp y) \join (y \imp x) \app 1;
$$
so it is a  {\bf G\"odel hoop} and
$\mathsf{GNPcL} \sse K(\mathsf{GH})$. For the converse if $\alg A \in K(\mathsf{GH})$, then $\alg A^-  \in \mathsf{GH}$. This means that for any $a,b \in A$
$$
((a \meet 1) \imp_1 (b\meet 1)) \join ((b \meet 1) \imp_1 (a \meet 1)) = 1.
$$
Now an application of Lemma \ref{eqn} does the trick.

\section{Admissible subalgebras}\label{admissub}

We observe that it is perfectly possible for two non isomorphic K-lattices to share the same negative cone; therefore if ${\alg A}\in\mathsf{KL}$ and $\alg B$ is a subalgebra of $\alg A$, we say that $\alg B$ is an \textbf{admissible subalgebra} if  $B^-=A^-$.

We want to study all the admissible subalgebras of $K({\alg L})$, for ${\alg L}\in\mathsf{CIRL}$. Since $K(\alg L)^- = \alg L\times\{1\}\cong \alg L$, an admissible subalgebra $\alg B$ of $K(\alg L)$ is a subalgebra such that $\alg B^- = \{(a,1): a \in L\}$.
Note that since the intersection of any family of admissible subalgebras of $\alg A$ is admissible, it makes sense to talk about the admissible subalgebra generated by a set $S \sse A$. It is obvious that the admissible subalgebra generated by $S\subseteq K(L)$ coincides with the subalgebra of $K(\alg L)$ generated by $S \cup L \times \{1\}$.

There is a general strategy to tackle this kind of problems; since any algebra in $\mathsf{CIRL}$ is the ordinal sum of sum irreducible algebras, it makes sense to first investigate admissible subalgebras of $K(\alg A)$ where $\alg A$ is sum irreducible, and then use this information to deduce results about the ordinal sums of sum irreducible algebras. Though we have no complete description of the sum irreducible algebras in $\mathsf{CIRL}$,  we do know a large class of sum irreducible algebras: if $\alg A$ is involutive (hence bounded by $0$), then $\alg A$ is sum irreducible. In fact suppose that $\alg A= \alg B \oplus \alg C$; then $0 \in B$ and for $c \in C\setminus \{1\}$, $(c \imp 0) \imp 0 = 0$. If $\alg A$ is involutive, this implies $c=0$ and hence $\alg C=\{1\}$, i.e. $\alg A$ is sum irreducible.

If $\alg A$ is involutive,  for $a,b \in A$ we define $a \oplus b = (a \imp 0) \imp b$. The following theorem is a slight modification of Theorem 6.11 in \cite{BusanicheCignoli2014}, implicit in that paper.

\begin{theorem}\label{kalmanwajsberg} Let $\alg A \in \mathsf{CIRL}$ be  involutive; then there is a one to one and onto correspondence between the lattice filters of $\alg A$ and the admissible subalgebras of $K(\alg A)$. More precisely, if $F$ is a lattice filter of $\alg A$, then
$$
K(\alg B,F) =\{(a,b) \in K(\alg B): a \oplus b \in F\} \le K(\alg B)
$$
is admissible. Conversely if $\alg A\in \mathsf{KL}$ satisfies that $\alg A^-$ is involutive, then
$$
F= \{((a \meet 1) \imp_1 0) \imp_1 \nneg a: a \in A\}
$$
is a lattice filter of $\alg A^-$ and $K(\alg A^-,F)$ is an admissible subalgebra of $K(\alg A^-)$.
\end{theorem}

The next step is to get information about the ordinal sum. First we observe:

\begin{lemma}\label{adm1}Let $\alg A,\alg B \in \mathsf{CIRL}$ and suppose that $\alg A \oplus \alg B$ exists. If $\alg T$ is an admissible subalgebra of $K({\alg A}\oplus {\alg B})$, then $T$ contains all the elements of the form $(a,b),(b,a),(b,b')$, for $a\in A\setminus\{1\}$ and $b,b'\in B$.\end{lemma}
\begin{proof}Take $a\in A\setminus\{1\}, b,b'\in B$. As $(b,1),(1,a)\in T$, then $(b,a)=(b,1)\cdot(1,a)\in T$, and also $(a,b)\in T$. Additionally, $(b,b')=(b,a)\meet (1,b')\in T$. \end{proof}

Next:

\begin{lemma}\label{adm2} Let $\alg A,\alg B \in \mathsf{CIRL}$ and suppose that $\alg A \oplus \alg B$ exists. If $\alg S$ is an admissible subalgebra of $K({\alg A})$, then ${\alg T}$, the admissible subalgebra of $K({\alg A}\oplus {\alg B})$ generated by $\alg S$, satisfies $S=T\cap (A\times A)$.
\end{lemma}
\begin{proof}   We only need to show that, if $(a,a')\in T\cap (A\times A)$, then $(a,a')\in S$. By what we said above $\alg T$ is the subalgebra generated by $S \cup B\times \{1\}$ so the generic element of $T$ has the form $t=t(\la a_1,a'_1\ra,\dots,\la a_n,a'_n\ra,\la b_1,1\ra,\dots,\la b_m,1\ra)$, where $\la a_i,a'_i\ra \in S$ for $i=1,\dots,n$ and $b_j \in B$ for $j=1,\dots,m$.  Hence we must prove that, if $t\in T\cap (A\times A)$, then $t\in S$.

The proof is by induction on the formation of $t(\vuc xn,\vuc yn)$. Suppose that $\la a,a'\ra \in S\setminus\{\la 1,1\ra\}$ and   and $b \in B\setminus\{1\}$:
\begin{enumerate}
\ib if $t = x \imp y$, then
$$
t(\la a,a'\ra,\la b,1\ra) = \la (a \imp b) \meet (1 \imp a'),a\cdot 1\ra = \la a',a\ra \in S;
$$
\ib if $t = y \imp x$, then if $a'\neq 1$
$$
t(\la a,a'\ra,\la b,1\ra) = \la (b \imp a) \meet  (a' \imp 1),ba' \ra = \la a,a'\ra \in S;
$$
and if $a'=1$,
$$
t(\la a,a'\ra,\la b,1\ra) = \la (b \imp a) \meet  (a' \imp 1),ba' \ra = \la a,b\ra \not\in A\times A;
$$
\ib if $t = xy$ then if $a,a'\neq 1$
$$
t(\la a,a'\ra,\la b,1\ra) =\la ab, (a \imp 1) \meet (b \imp a')\ra = \la a,a'\ra \in S;
$$
if $a=1,a'\neq 1$
$$
t(\la a,a'\ra,\la b,1\ra) =\la ab, (a \imp 1) \meet (b \imp a')\ra = \la b,a'\ra \not\in A\times A;
$$
and if $a\neq 1,a'=1$
$$
t(\la a,a'\ra,\la b,1\ra) =\la ab, (a \imp 1) \meet (b \imp a')\ra = \la a,b\ra \not\in A\times A;
$$
\ib if $t= x\join y$, then if $a\neq 1$
$$
t(\la a,a'\ra,\la b,1\ra) = \la a \join b, a' \meet 1\ra = \la b,a'\ra \notin A \times A;
$$
and if $a=1$
$$
t(\la a,a'\ra,\la b,1\ra) = \la a \join b, a' \meet 1\ra = \la 1,a'\ra \in S;
$$
\ib if $t=x \meet y$, then if $a\neq 1$
$$
t(\la a,a'\ra,\la b,1\ra) = \la a \meet b, a'\join 1\ra  = \la a,1\ra \in S.
$$
and if $a=1$
$$
t(\la a,a'\ra,\la b,1\ra) = \la a \meet b, a'\join 1\ra  = \la 1,1\ra \in S.
$$
\end{enumerate}
This takes care of the  base step of the induction; the induction step is similar and it is left as an exercise to the reader.
\end{proof}

As a consequence of Lemmas \ref{adm1} and \ref{adm2} we get:

\begin{theorem}\label{admissibleordinalsum} The admissible subalgebras of $K({\alg A}\oplus {\alg B})$ are in one to one correspondence with the admissible subalgebras of $K({\alg A})$. Moreover, if $\alg S$ is an admissible subalgebra of $K({\alg A})$, then $T_S^\alg B=S\cup (A\times B)\cup (B\times A)\cup (B\times B)$ is the universe of an admissible subalgebra of $K({\alg A}\oplus {\alg B})$.
And if $\alg T$ is an admissible subalgebra of $K({\alg A}\oplus {\alg B})$, then $S_T=T\cap A\times A$ is the universe of an admissible subalgebra of $K({\alg A})$ that satisfies $T_{S_T}^\alg B=T$.
\end{theorem}

We have the following lemma, whose proof is straightforward.

\begin{lemma}\label{admissiblelemma} Let $\alg A, \alg B, \alg C \in \mathsf{CIRL}$ and let $\alg S,\alg S'$ be admissible subalgebras of $K(\alg A)$. Then:
\begin{enumerate}
\item $\alg S \le \alg T_S^\alg B$;
\item if $\alg S' \le \alg S$ and $\alg B \le \alg C$ then $\alg T_S^\alg B \le \alg T_{S'}^\alg C$.
\end{enumerate}
\end{lemma}

Finally, since every algebra is a subdirect product of subdirectly irreducible algebras it is important to understand what are the admissible subalgebras of a subdirect product of algebras.

\begin{theorem}\label{subdirect} Let $\alg A  \in \mathsf{CIRL}$ ; if $(\th_i)_{i \in I}$  induces a subdirect decomposition of ${\alg A}$, then the admissible subalgebras of $K({\alg A})$ are exactly the subdirect products of admissible subalgebras of $K({\alg A}/\th_i)$.
\end{theorem}
\begin{proof} If $\alg S$ is an admissible subalgebra of $K(\alg A)$, and $\hat\theta _i$ is the congruence corresponding to $\theta$ in the isomorphism of Theorem \ref{congruences}, then it is clear that
$(\hat \theta_i \cap S^2)_{i \in I}$ is a subdirect decomposition of $\alg S$  and the universe of $\alg S/\hat\theta_i \cap S^2$ is
$$
S/\hat\th_i =\{u/\hat\theta_i: (u,s) \in \hat\theta_i\}.
$$
Clearly $S/\hat\theta_i$ is the universe of a subalgebra of $K(\alg A)/\hat\theta_i$ and it is also compatible, since $(a,1) \in S$ for all $a \in  A$ implies $(a/\th_i,1/\th_i) = (a,1)/\hat\theta_i \in S/\hat\theta_i$.

Conversely suppose that $\alg S$ is a subdirect product of admissible subalgebras of $K(\alg A/\th_i) = K(\alg A)/\hat\theta_i$; this means that there are $\alg T_i \le K(\alg A)/\hat\theta_i$ that are admissible and  moreover $u \longmapsto (u/\hat\theta_i)_{i \in I}$ is an embedding of $\alg S$ in $\Pi_{i\in I} \alg T_i$. But since each $\alg T_i$ is admissible, if $a \in A$ then $(a,1)/\hat\theta_i \in T_i$ for $i \in  I$; so $((a,1)/\hat\theta_i)_{i \in I} \in \Pi_{i\in I} \alg T_i$ and so $(a,1) \in S$. Since this holds for any $a \in A$, we conclude that $\alg S$ is admissible as well.
\end{proof}

Let us quote another result in same line. Let $\alg B$ be a  Brouwerian lattice; an element of $\alg B$ is {\bf dense} if it is of the form $a \join (a \imp b)$ for some $a,b \in B$.   A filter $F$ of $\alg B$ is {\bf regular} of it contains all the dense elements; it is easy to see that $F$ is regular if and only if $\alg B/F$ is a generalized boolean algebra. The following is again a modification of Theorems 3.3 and 3.5 in \cite{ABGM2017}.

\begin{theorem}\label{kalmanheyting} Let $\alg B$ be a  Brouwerian lattice; then there is a one to one and  correspondence between the regular filters of $\alg B$ and the algebras in $\mathsf{KL}$ whose negative cone is isomorphic with $\alg B$. More precisely, if $F$ is a regular  filter of $\alg B$, then
$$
K(\alg B,F)= \{(a,b) \in K(\alg B):  a \join b \in F\} \le K(\alg B)
$$
and $K(\alg B,F)^- \cong \alg B$. Conversely if $\alg A \in\mathsf{NPcL}$  then  $\alg A^-$ is a Brouwerian lattice and
$F=\{(a \join \nneg a) \meet 1: a \in A\}$ is a regular filter of $\alg A^-$ such that $K(\alg A^-,F) \cong \alg A$.
\end{theorem}

\section{Atoms}\label{section3}

Are there subvarieties of $\mathsf{KL}$ that are not Kalman varieties? Of course there are and here is a simple example; let $\mathbf 2$ be the two element Boolean algebra and consider $K(\mathbf 2)$. If $2 = \{0,1\}$ then $K(\mathbf 2)$ has four elements  and it is the only four element K-lattice, so it makes sense to denote it by $\alg K_4$.  Moreover  $\{(0,1),(1,1),(1,0)\}$ is the universe of a subalgebra $\alg K_3$ of $K(\mathbf 2)$ (in \cite{BusanicheCignoli2014} it is called $\alg P_3$). In \cite{BusanicheCignoli2014} it is also shown that $\alg K_3$ is the only totally ordered algebra in $\mathsf{KL}$ and therefore
$\mathsf{PKL}= \VV(\alg K_3)$ is the only prelinear subvariety of $\mathsf{KL}$. However $\mathsf{PKL}$ cannot be Kalman since $\alg K_3^- = \mathbf 2$; so $K(\mathbf 2) \in K(\mathsf{PKL}^-)$ and $K(\mathbf 2)$ is a simple algebra (since $\mathbf 2$ is simple) that is not totally ordered. Hence $K(\mathbf 2) \notin \mathsf{PKL}$ and the inclusion is proper.

In this section we will investigate the bottom of the lattice $\Lambda(\mathsf{KL})$. Note that $\alg K_3$ is strictly simple, i.e. it is a simple algebra with no proper subalgebras; this is enough (see \cite{Galatos2005}) to guarantee that $\mathsf{PKL}$ is an atom $\Lambda(\mathsf{KL})$.  As a matter of fact:

\begin{theorem}\label{onlyatom} $\VV(\alg K_3)$ is the only finitely generated atom in $\Lambda(\mathsf{KL})$.
\end{theorem}
\begin{proof} First suppose that $\alg A$ is a finite non-trivial algebra in $\mathsf{KL}$; then $\alg A$ has a bottom element  $0\in A^-$ and (since it is $1$-involutive) also  a top element $\top= \nneg 0$. Consider the set $\{0,1,\top\}$; we claim that it is the universe of a subalgebra of $\alg A$ clearly isomorphic with $\alg K_3$.  Since in a $1$-involutive lattice $a \imp b = \nneg(a\cdot \nneg b)$ we have only to check closure under $\nneg$ and $\cdot$. The first is obvious, and for the second $0^2 = 0$, $\top^2=\top$ and $0\top=0$ hold in any bounded residuated lattice, so clearly $\alg K_3 \le \alg A$.

Now a finitely generated atom must be generated by a finite strictly simple algebra $\alg A$; by the above $\alg K_3 \le \alg A$ and so $\VV(\alg K_3) = \VV(\alg A)$.
But J\`onnson's Lemma implies that $\alg K_3 \cong \alg A$.
\end{proof}

What about not finitely generated atoms in $\mathsf{KL}$?  It is clear that if $\vv V$ is any subvariety of $\mathsf{CIRL}$ containing finite models, then $\mathsf{PKL} \sse K(\vv V)$. It follows that we should start from a variety that has no finite models and sits sufficiently low in the lattice of subvarieties of $\mathsf{CIRL}$. A commutative residuated lattice is {\bf cancellative} if its underlying monoid is cancellative; it is clear that no bounded residuated lattice can be cancellative and so no finite one can be cancellative. Cancellative commutative residuated lattices have been investigated in \cite{BahlsColeGalatos2003}; they form a variety (both in the integral and non integral case) axiomatized by
$$
x \imp xy \app y
$$
and the two varieties are denoted by $\mathsf{CanCRL}$ and $\mathsf{CanCIRL}$ respectively. There are exactly two cancellative lattices that are atoms in the lattice of subvarieties of $\mathsf{CRL}$, i.e. $\VV(\mathbb Z)$ and $\VV(\mathbb Z^-)$, where $\mathbb Z$ are the integers regarded as a an $\ell$-group \cite{BahlsColeGalatos2003}. Also $\VV(\mathbb Z^-)$ is the only cancellative atom in $\mathsf{CIRL}$ and $\VV(\mathbb Z)$ is a subvariety of $\mathsf{CRL}$ that is distributive and
$1$-involutive. However it does not satisfy (K1) so it is not a variety of K-lattices.

Let's look at $\mathbb Z^-$; we prefer to work with its {\em multiplicative equivalent} $\alg C_\o$. Let $C^\o$ be the universe of the free monogenerated monoid (say generated by $a$) ordered as a chain $1=a^0 > a > a^2>\dots > a^n > \dots$; let define
$$
a^ra^s = a^{r+s}\qquad a^r \imp a^s = a^{\max\{s-r,0\}}.
$$
Then $\alg C_\o$ is a totally ordered Wajsberg hoop isomorphic with $\mathbb Z^-$.  The class $\II\SU\PP_u(\alg C_\o)$ has been investigated in \cite{AglianoMontagna2003}; it turns out that it consists of all totally ordered cancellative hoops; moreover if $\alg A$ is a totally ordered cancellative hoop then $\II\SU\PP_u(\alg A) = \II\SU\PP_u(\alg C_\o)$ \cite{AglianoMontagna2003}. This implies that the variety $\mathsf C$ of cancellative hoops is equal to  $\VV(\alg A)$ for any totally ordered cancellative hoop $\alg A$ (in particular $\mathsf{C}=\VV(\mathbb Z^-)$). Next we need a lemma.

\begin{lemma}\label{tocancellative} Let $\alg A \in \mathsf{KL}$ such that $\alg A^-$ is cancellative. Then $\alg A \cong K(\alg A^-)$.
\end{lemma}
\begin{proof} We already know that $f:a \longmapsto (a \meet 1, \nneg a \meet 1)$ is an embedding of $\alg A$ in $K(\alg A^-)$ (Lemma \ref{inclusions}) so we only need to show that it is surjective.
If $a,b \in A^-$, then of course $ab \in A^-$ and $a \cdot \nneg ab \in A$; let's evaluate
\begin{align*}
f(a\cdot \nneg ab) &= f(a) \cdot f(\nneg ab) = f(a) \cdot \nneg f(ab) \\
&= (a \meet 1, \nneg a \meet 1) \cdot (\nneg ab \meet 1, ab \meet 1) = (a,1) \cdot(1, ab) \\
&= (a, (a \imp ab) \meet 1) = (a, b \meet 1) = (a,b)
\end{align*}
Since $a,b$ where generic we have shown that $f$ is onto  $K(\alg A^-)$.
\end{proof}

\begin{theorem}\label{kcancellative}  The variety $K(\mathsf C)= \VV(K(\alg C_\o))$ is a (non finitely generated) atom in $\Lambda(\mathsf{KL})$.
\end{theorem}
\begin{proof} By J\'onsson Lemma every subdirectly irreducible algebra in $\VV(K(\alg C_\o))$ is in $\HH\SU\PP_u(K(\alg C_\o)) \sse \SU K(\HH\SU\PP_u(\alg C_\o))$.
Now every algebra in  $\HH\SU\PP_u(\alg C_\o)$ is a totally ordered cancellative hoop, hence if $\alg A \in \HH\SU\PP_u(K(\alg C_\o))$ then $\alg A\le K(\alg B)$ where $\alg B$ is a totally ordered cancellative hoop.
It follows that $\alg A^- \le \alg B$ and so $\alg A^-$ is a totally ordered cancellative hoop as well; now $\alg C_\omega$ is clearly embeddable in any totally ordered cancellative hoop (any one generated subalgebra will do)
so $\alg C_\o \le \alg A^-$ and thus, using Lemma \ref{tocancellative},  $K(\alg C_\o) \le K(\alg A^-) =\alg A$.  Hence $\VV(K(\alg C_\o)) \sse \VV(\alg A)$; this proves that $K(\vv C)$ is an atom in $\Lambda(\mathsf{KL})$.
\end{proof}

We observe in passing that $K(\mathsf{C})$ is axiomatized by
$$
(x \meet 1) \imp_1 xy \app y \meet 1
$$

We do not know if there are other atoms in $\Lambda(\mathsf{KL})$ but we suspect that there are in fact more. However it is clear that $\mathsf{PKL}$ and $K(\mathsf{C})$ are the only atoms in $\Lambda(K(\vv H))$, i.e. the variety of K-lattices whose negative cones are hoops (more on this topic later).

The next step is to look at almost minimal varieties; a variety $\vv V$ is {\bf almost minimal} in $\Lambda(\mathsf{KL})$ if it covers an atom.
What are the finitely generated almost minimal varieties above $\mathsf{PKL}$? One is very easy to find: the variety of  {\bf generalized Boolean algebras} is the variety $\mathsf{GBA}$ whose members are the zero-free subreducts of Boolean algebras (that such class is a variety follows from general facts about subvarieties of $\mathsf{CIRL}$). The following lemma is useful:

\begin{lemma} \cite{Galatos2005} \begin{enumerate}
\item $\mathsf{GBA}$ is the variety of residuated lattices in which every principal filter is polynomially equivalent to a Boolean algebra;
\item $\mathsf{GBA}$ is the variety of Brouwerian lattices satisfying
$$
(x \imp y) \imp y \app (y \imp x)\imp x.
$$
\end{enumerate}
\end{lemma}

\begin{lemma}  $\VV(\alg K_4)$ is a Kalman variety, more precisely  $\VV(\alg K_4) = K(\mathsf{GBA})$, is contained in $\mathsf{NPcL}$  and it is almost minimal. $V(\mathsf{GBA})$ is axiomatized, relative to $\mathsf{NPcL}$ by the equation:
$$
((x\meet 1) \imp_1 y) \imp_1 y \app ((y\meet 1) \imp_1 x) \imp_1 x.
$$
\end{lemma}
\begin{proof} The first claim is obvious by Lemma \ref{Kproperties} and the fact that $\mathsf{GBA} = \VV(\mathbf 2)$.
For the second  by  J\'onsson's Lemma  the subdirectly irreducible  algebras in $\VV(\alg K_4)=\VV(K(\mathbf 2))$ are among the homomorphic images of subalgebras of $K(\mathbf 2)$; but the only nontrivial subalgebra of $K(\mathbf 2)$ is $\alg K_3$ that is simple. Since $K(\mathbf 2)$ is itself simple the only subdirectly irreducible algebras in $\VV(K(\mathbf 2))$ are $K(\mathbf 2)$ and $\alg K_3$. Hence $K(\mathsf{GBA})$ is a finitely generated almost minimal above $\mathsf{PKL}$.

Finally  if $\alg A \in \mathsf{NPcL}$ satisfies the above equation, then $\alg A^- \in \mathsf{GBA}$; so $\alg A \le K(\alg A^-) \in K(\mathsf{GBA})$. For the converse we simply invoke Lemma \ref{eqn}.
\end{proof}

The fact that $\mathsf{CIRL}$ has the CEP  gives us a strategy to find other finitely generated almost minimal varieties above $\mathsf{PKL}$.
We start with a simple fact whose proof is left to the reader.

\begin{lemma}\label{sub} Let $\alg A\in \mathsf{KL}$ satisfy $A\subseteq K(L)$ for some $\alg L\in \mathsf{CIRL}$ and let $\alg C \le \alg L$. Then $B= (C \times C) \cap A$ is the universe of a subalgebra $\alg B$ of $\alg A$.
\end{lemma}

An algebra $\alg A \in \mathsf{CIRL}$ is {\bf tight} if
\begin{enumerate}
\ib $|A|> 2$;
\ib $\alg A$ is bounded by $0$ and any element different from $0,1$ generates $\alg A$.
\end{enumerate}
Tight algebras have been investigated in \cite{AglianoGalatosMarcos2020}; it turns out that they generate almost minimal varieties in $\Lambda(\mathsf{CIRL})$.
If $\alg A$ is bounded by $0$, we can define the {\bf order} of $a \in A$ as follows:
$$
o(a) = \left\{
             \begin{array}{ll}
               \min\{n: a^n=0\}, & \hbox{if such $n$ exists;} \\
               \infty, & \hbox{otherwise.}
             \end{array}
           \right.
$$

The next Lemma can be obtained from the definition of tight algebras. A detailed proof can be found in \cite{AglianoGalatosMarcos2020}.

\begin{lemma}\label{techlemma1} \cite{AglianoGalatosMarcos2020} Let $\alg A$ be a tight algebra bounded by $0$. Then
\begin{enumerate}
\item $\alg A$ cannot have idempotents different from $0$ and $1$;
\item $\alg A$ is simple;
\item every element $a \ne 1$ must have finite order;
\item if the set $\{o(a): a \in A\setminus\{1\}\}$ has an upper bound, then  then $1$ is completely join irreducible, i.e. $\alg A$  has exactly one coatom.
\end{enumerate}
\end{lemma}

So in particular a finite tight algebra has a unique  coatom.
The next step is to describe how tight algebras help us find finitely generated almost minimal varieties above $\mathsf{PKL}$ different from $K(\mathsf{GBA})$.

\begin{lemma}\label{meetirred1} Let $\alg A \in \mathsf{CIRL}$ be bounded by 0 and such that $0$ is meet irreducible. Then $A\times A\setminus\{(0,0)\}$ is the universe of an admissible subalgebra of $K(\alg A)$.\end{lemma}
\begin{proof}It is sufficient to show that $(a,b)\meet(c,d)\neq(0,0)$ and $(a,b)\cdot(c,d)\neq(0,0)$ for $(a,b),(c,d)\in A\times A\setminus\{(0,0)\}$.
\begin{itemize}
	\item If $(a,b)\meet(c,d)=(0,0)$, then $a\meet c=0$ and $b\join d=0$. The second one implies that $b=d=0$, and the first that $a=0$ or $c=0$, as $0$ is meet irreducible. Therefore $(a,b)=(0,0)$ or $(c,d)=(0,0)$.
	\item  If $(a,b)\cdot(c,d)=(0,0)$, then $ac=0$ and $a\imp d\meet c\imp b=0$. The second one implies that $a\imp d=0$ or $c\imp b=0$, as $0$ is meet irreducible. Assuming $a\imp d=0$, from $ac=0$ we have that $c\leq a\imp 0\leq a\imp d=0$, so $c=0$, and as $d\leq a\imp d=0$ we also obtain $d=0$. Therefore $(c,d)=(0,0)$.
\end{itemize}
\end{proof}

\begin{lemma}\label{meetirred2} Let $\alg A \in \mathsf{CIRL}$ be a tight algebra bounded by 0 and such that $0$ is meet irreducible. Then there exists an admissible subalgebra of $K(\alg A)$ with no subalgebra isomorphic to $\alg K_4$.\end{lemma}
\begin{proof}Assume $\alg B$ is an admissible subalgebra of $K(\alg A)$ such that there exists a subalgebra of $\alg B$ isomorphic to $\alg K_4$. This subalgebra has universe
$$
\{(a,1),(1,1),(1,a),(a,a)\}
$$
for some $a\in A$, and therefore $\{a,1\}$ is the universe of a subalgebra of $\alg A$. As $\alg A$ is tight, we have that $a=0$ and $(0,0)\in B$. By Lemma \ref{meetirred1}, there exists an admissible subalgebra $\alg C$ with universe $A\times A\setminus\{(0,0)\}$, and no subalgebra of $\alg C$ is isomorphic to $\alg K_4$.\end{proof}

\begin{theorem}\label{cover}  Let $\alg A \in \mathsf{KL}$ be a finite algebra that generates an almost minimal variety above $\mathsf{PKL}$, different from $\VV(\alg K_4) = K(\mathsf{GBA})$. Then $\alg A$ is simple and $\alg A^-$ is tight.

On the other hand, let $\alg A$ be a finite tight algebra in $\mathsf{CIRL}$ bounded by $0$ and such that $0$ is meet irreducible. Then $K(\alg A)$ is simple and has a subalgebra that generates an almost minimal variety above $\mathsf{PKL}$ different from $K(\mathsf{GBA})$.
\end{theorem}
\begin{proof}
Suppose that $\alg A$ is a finite algebra generating an almost minimal variety above $\mathsf{PKL}$ different from $K(\mathsf{GBA})$. Clearly $|A^-|> 2$, and any proper subalgebra of $\alg A$ must be isomorphic with $\alg K_3$. This implies that $\alg A^-$ cannot have proper subalgebras different from $\{0,1\}$, so it is tight. By Lemma \ref{techlemma1}  $\alg A^-$ is simple and so,
by Lemma \ref{congruences}, $\alg A$ is simple as well.

Conversely, let $\alg A$ be a finite tight algebra such that $0$ is meet irreducible. Since $\alg A$ is finite and simple, so is $K(\alg A)$. Moreover since the only proper nontrivial subalgebra is $\{0,1\}$, then any subalgebra $\alg B$ of $K(\alg A)$ such that $\alg B^- \ne \alg A$ must be $\alg K_3$ or $\alg K_4$. Now, consider the minimal admissible subalgebra $\alg C$ of $K(\alg A)$. By Lemma \ref{meetirred2}, the only nontrivial subalgebras of $\alg C$ are $\alg C$ and $\alg K_3$. Since $K(\alg A)$ is simple, also $\alg C$ is simple by the CEP. By J\'onsson's Lemma $\VV(\alg C)$ is almost minimal above $\VV(\alg K_3) = \mathsf{PKL}$.
\end{proof}

Theorem \ref{cover} gives us a way of finding finitely generated almost minimal varieties above $\mathsf{PKL}$; we can start from a finite tight algebra  $\alg A \in \mathsf{CIRL}$ and work our way through $K(\alg A)$. We will illustrate that in the next section.

\section{Almost minimal varieties}\label{almostmin}

In this section we describe several almost minimal varieties in $\mathsf{KL}$, using the results of the previous sections. Let's start with tight algebras;  we have a lemma whose proof is straightforward.

\begin{lemma} Let $\alg A$ be a tight algebra in $\mathsf{CIRL}$; if $\alg B$ is a subalgebra of $K(\alg A)$, then either $\alg B \cong \alg K_3$ or $\alg B \cong \alg K_4$ or else $\alg B$ is an admissible subalgebra.
\end{lemma}

It is clear from the definition that a finite Wajsberg hoop is tight if and only if it is isomorphic with $\alg \L_p$ for some prime $p$. Since it is a chain, Theorem \ref{cover} applies; moreover from the above lemma it follows that  we can consider only the admissible subalgebras of $K(\alg \L_p)$  and they can be found  using Theorem \ref{kalmanwajsberg}: there are exactly $p$ proper lattice ideals in $\alg \L_p$ and they give raise to $p$ non-isomorphic admissible subalgebras of $K(\alg \L_p)$. Moreover, the minimal admissible subalgebra of $K(\alg \L_p)$ has universe $\{(u,v): u \oplus v \ge a^0=1\}$. A simple combinatorial argument shows that this algebra has $\sum_{i=1}^{p+1} i$ elements and we denote it by $\alg K_{0,p}$. It follows that $\VV(\alg K_{0,p})$ is almost minimal above $\mathsf{PKL}$ for every prime $p$; since if $p\ne q$ $\alg K_{0,p} \not \cong \alg K_{0,q}$ these varieties are all distinct and  they are the only almost minimal varieties whose negative cones are Wajsberg hoops. Of course, for any $n \le p$, the set $\{(u,v): u \oplus v \ge a^n\}$ is a universe of a subalgebra of $K(\alg \L_p)$, that we denote by $\alg K_{n,p}$.
Observe that $\alg K_{p,p}= K(\alg \L_p)$; moreover a straightforward application of J\'onsson Lemma yields that $\VV(\alg K_{r+1,p})$ covers $\VV(\alg K_{r,p})$ for
$r < p$. In Figure \ref{p=3} we see all proper subalgebras in  case  $p=3$.

\begin{figure}[htbp]
\begin{center}
\begin{tikzpicture}[scale=.5]
\draw (0,2) -- (3,5) -- (0,8) -- (-2,6) -- (-1,5) -- (-2,4)-- (0,2);
\draw (1,3) -- (-1,5);
\draw (2,4) -- (-1,7);
\draw (-1,3) -- (2,6) ;
\draw (-1,5) -- (1,7);
\draw[fill] (0,2) circle [radius=0.05];
\draw[fill] (0,4) circle [radius=0.05];
\draw[fill] (0,6) circle [radius=0.05];
\draw[fill] (0,8) circle [radius=0.05];
\draw[fill] (1,3) circle [radius=0.05];
\draw[fill] (1,5) circle [radius=0.05];
\draw[fill] (1,7) circle [radius=0.05];
\draw[fill] (2,4) circle [radius=0.05];
\draw[fill] (2,6) circle [radius=0.05];
\draw[fill] (3,5) circle [radius=0.05];
\draw[fill] (-1,3) circle [radius=0.05];
\draw[fill] (-1,5) circle [radius=0.05];
\draw[fill] (-1,7) circle [radius=0.05];
\draw[fill] (-2,4) circle [radius=0.05];
\draw[fill] (-2,6) circle [radius=0.05];
\node at (2,2) {$\alg K_{2,3}$};
\draw (8,2) -- (11,5) -- (8,8) -- (7,7) -- (8,6) -- (7,5) -- (8,4) -- (7,3) -- (8,2);
\draw (9,7) -- (8,6) -- (10,4) -- (9,3) -- (8,4) -- (10,6);
\draw[fill] (8,2) circle [radius=0.05];
\draw[fill] (8,4) circle [radius=0.05];
\draw[fill] (8,6) circle [radius=0.05];
\draw[fill] (8,8) circle [radius=0.05];
\draw[fill] (9,3) circle [radius=0.05];
\draw[fill] (9,5) circle [radius=0.05];
\draw[fill] (9,7) circle [radius=0.05];
\draw[fill] (10,4) circle [radius=0.05];
\draw[fill] (10,6) circle [radius=0.05];
\draw[fill] (11,5) circle [radius=0.05];
\draw[fill] (7,3) circle [radius=0.05];
\draw[fill] (7,5) circle [radius=0.05];
\draw[fill] (7,7) circle [radius=0.05];
\node at (10,2) {$\alg K_{1,3}$};
\draw (16,2) -- (19,5) -- (16,8)-- (17,7)-- (16,6) -- (17,5) -- (16,4) -- (17,3) -- (16,2);
\draw (18,6) -- (17,5) -- (18,4);
\draw[fill] (16,2) circle [radius=0.05];
\draw[fill] (16,4) circle [radius=0.05];
\draw[fill] (16,6) circle [radius=0.05];
\draw[fill] (16,8) circle [radius=0.05];
\draw[fill] (17,3) circle [radius=0.05];
\draw[fill] (17,5) circle [radius=0.05];
\draw[fill] (17,7) circle [radius=0.05];
\draw[fill] (18,4) circle [radius=0.05];
\draw[fill] (18,6) circle [radius=0.05];
\draw[fill] (19,5) circle [radius=0.05];
\node at (18,2) {$\alg K_{0,3}$};
\end{tikzpicture}
\end{center}
\caption{The lattice structure of $\alg K_{n,3}$ \label{p=3}}
\end{figure}
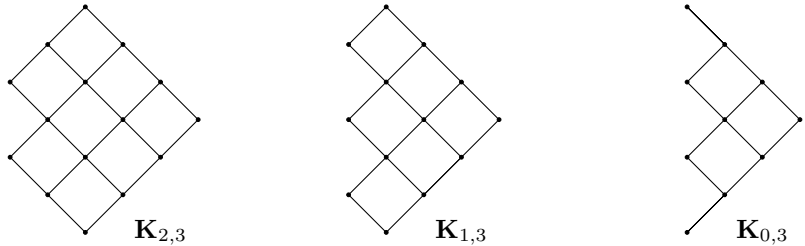

{
For an example of a different almost minimal variety obtained via a tight algebra that is not Wajsberg, we can use the tight chain $\alg C_5$ described in \cite{AglianoGalatosMarcos2020}. It is a five element chain that is ``almost Wajsberg'', given by
\begin{align*}	0=a^3 < a^2 < a \imp 0 < a < 1\end{align*}
The reader can easily check that is indeed tight. Easy calculations show that the subalgebra of $K(\alg C_5)$ generating a new almost minimal variety above $\mathsf{PKL}$ is the one in Figure \ref{newcover}.

\begin{figure}[htbp]
\begin{center}
\begin{tikzpicture}[scale=.5]
\draw (0,2) -- (4,6) -- (0,10) --  (-1,9) -- (0,8) --(-1,7) -- (0,6) --(-1,5) -- (0,4) -- (-1,3) -- (0,2);
\draw (0,8) -- (3,5);
\draw (0,6) --  (2,4);
\draw (0,4) -- (1,3);
\draw (0,4) -- (3,7);
\draw (0,6) -- (2,8);
\draw (0,8) -- (1,9);
\draw[fill] (0,2) circle [radius=0.05];
\draw[fill] (0,4) circle [radius=0.05];
\draw[fill] (0,6) circle [radius=0.05];
\draw[fill] (0,8) circle [radius=0.05];
\draw[fill] (0,10) circle [radius=0.05];
\draw[fill] (-1,3) circle [radius=0.05];
\draw[fill] (-1,5) circle [radius=0.05];
\draw[fill] (-1,7) circle [radius=0.05];
\draw[fill] (-1,9) circle [radius=0.05];
\draw[fill] (1,3) circle [radius=0.05];
\draw[fill] (1,5) circle [radius=0.05];
\draw[fill] (1,7) circle [radius=0.05];
\draw[fill] (1,9) circle [radius=0.05];
\draw[fill] (2,4) circle [radius=0.05];
\draw[fill] (2,6) circle [radius=0.05];
\draw[fill] (2,8) circle [radius=0.05];
\draw[fill] (3,5) circle [radius=0.05];
\draw[fill] (3,7) circle [radius=0.05];
\draw[fill] (4,6) circle [radius=0.05];
\node[right] at (0,2) {\footnotesize  $(0,1)$};
\node[right] at (1,3) {\footnotesize  $(a,1)$};
\node[right] at (2,4) {\footnotesize  $(b,1)$};
\node[right] at (3,5) {\footnotesize  $(a,1)$};
\node[right] at (4,6) {\footnotesize  $(1,1)$};
\node[right] at (3,7) {\footnotesize  $(1,a)$};
\node[right] at (2,8) {\footnotesize  $(1,b)$};
\node[right] at (1,9) {\footnotesize  $(1,a^2)$};
\node[right] at (0,10) {\footnotesize  $(1,1)$};
\end{tikzpicture}
\end{center}
\caption{The subalgebra of $K(\alg C_5)$ generating a new cover of $\mathsf{PKL}$.\label{newcover}}
\end{figure}
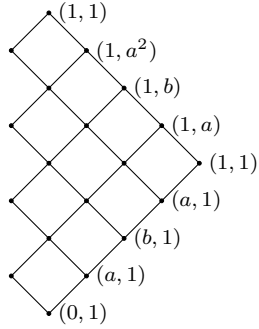

We now turn to almost minimal varieties covering the non finitely generated atom $K(\vv C)$.  By J\'onnson Lemma and Theorem \ref{kcancellative} $\VV(\alg K_3, K(\alg C_\o))$ is almost minimal and covers both atoms. Another interesting almost minimal subvariety can be obtained in the following way: recall that a basic hoop is a prelinear divisible hoop and that Wajsberg hoops,  hence cancellative hoops, are prelinear. It follows that any variety of basic hoops is generated by totally ordered hoops.  Suppose that $\vv V,\vv W$ are two prelinear subvarieties of $\mathsf{CIRL}$; we define
$$
\mathsf{V} \oplus^t \vv W = \VV(\{\alg A \oplus \alg B: \alg A \in \vv V,\alg B \in \vv W\ \text{and both totally ordered}\}).
$$
This kind of varieties has been investigated in \cite{AglianoMontagna2003}; from the results therein we get that
\begin{enumerate}
\ib $\vv C \oplus^t \vv C= \VV(\alg C_\o \oplus \alg C_\o)$ \cite[Theorem 7.4]{AglianoMontagna2003};
\ib $\vv C \oplus^t \vv C$ is axiomatized (relative to basic hoops) by
$$
(( x \imp y) \imp y) \meet ((y \imp z) \imp z) \le x \join y \join z\qquad x \imp x^2 \app 1
$$
\cite[Theorem 5.1]{AglianoMontagna2003}
\ib $\HH\SU\PP_u (\alg C_\o \oplus \alg C_\o) = \HH\SU\PP_u(\alg C_\o) \cup \{\alg A \oplus \alg B: \alg A \in \SU\PP_u(\alg C_\o), \alg B \in \HH\SU\PP_u(\alg C_\o\})$ \cite[Lemma 7.1]{AglianoMontagna2003}, hence $\vv C \oplus \vv C$ is generated by any algebra of the form $\alg A \oplus \alg B$ where $\alg A,\alg B \in \vv C$, they are totally ordered and $\alg B$ is nontrivial.
\end{enumerate}
Again using J\'onsson Lemma we get that $\vv C \oplus^t \vv C$ covers $\vv C$ (and hence it is almost minimal in $\Lambda(\mathsf{WH})$).

\begin{lemma} The variety $K(\vv C\oplus^t \vv C)$ is almost minimal in $\Lambda(\mathsf{KL})$.
\end{lemma}
\begin{proof} We will show that $K(\vv C\oplus^t \vv C)$ covers $K(\vv C)$. Note that
$$
K(\vv C \oplus^t \vv C) = K(\VV(\alg C_\o \oplus \alg C_\o)) =\VV(K(\alg C_\o \oplus \alg C_\o));
$$
by J\'onsson Lemma all the subdirectly irreducible algebras are in $\HH\SU\PP_u(K(\alg C_\o \oplus \alg C_\o)) = K(\HH\SU\PP_u(\alg C_\o \oplus \alg C_\o))$, i.e.
are of the form $K(\alg A \oplus \alg B)$ with $\alg A,\alg B \in \vv C$. But $\alg A \oplus \alg B$ generates $\vv C$ is $\alg B$ is trivial and $\vv C \oplus^t \vv C$ if $\alg B$ is nontrivial. It follows that $K(\vv C \oplus^t \vv C)$ covers $K(\vv C)$.
\end{proof}

\section{Lattices of subvarieties}\label{latticeofsub}

First we improve a little on Lemma \ref[techlemma}.

\begin{lemma}\label{embedding} For any subvariety $\vv U$ of $\mathsf{CIRL}$ the operators $K$ and ${}^-$ form a left adjoint pair from $\Lambda(\vv U)$ to $\Lambda (K(\vv U))$, i.e.
for any $\vv V \in \Lambda(\vv U)$ and $\vv W \in \Lambda(K(\vv U))$
$$
K(\vv V) \sse \vv W\qquad\text{iff}\qquad \vv V \sse \vv W^-.
$$
Moreover $K: \Lambda(\vv V) \longrightarrow \Lambda(K(\vv V))$ is a complete lattice embedding.
\end{lemma}
\begin{proof}
The first part is Lemma \ref{techlemma}(7). The operator $K$ clearly respects infinite meets, since it a part of a
left adjoint pair. For the join,  by Lemma \ref{techlemma},
\begin{align*}
K(\bigvee_{i \in I}(\vv V_i) ) &=K(\VV(\bigcup_{i \in I} \vv V_i))\\
& =\VV(K(\bigcup_{i \in I} \vv V_i)) = \VV(\bigcup_{i\in I} K(\vv V_i)) = \bigvee_{i\in I} K(\vv V_i).
\end{align*}
\end{proof}

As a consequence of the  completeness of the lattice  embedding, we observe that, if we take a Kalman subvariety of $\mathsf{KL}$, say $K(\vv V)$, then there is a copy of $\Lambda(\vv V)$ inside $\Lambda(K(\vv V))$.

A variety $\vv V$ has the {\bf finite model property} (FMP) if its equational theory has the finite model property; this is equivalent to saying that $\vv V$ is generated by its finite algebras. A variety $\vv V$ has the {\bf finite embeddability property} (FEP) if for any algebra $\alg A \in \vv V$ and for any finite partial subalgebra $\alg B$ of $\alg A$, there is a finite algebra $\alg C \in \vv V$ in which $\alg B$ can be embedded. It is evident that if $\vv V$ is locally finite, then it has the FEP and if $\vv V$ has the FEP, then it has the FMP.

\begin{theorem} \label{FEP} Let $\vv V$ be a subvariety of $\mathsf{CIRL}$:
\begin{enumerate}
\item $\vv V$  is locally finite if and only if  $K(\vv V)$ is locally finite;
\item $\vv V$ has the FEP if and only if  $K(\vv V)$ has the FEP;
\item $\vv V$ has the FMP if and only if $K(\vv V)$ has the FMP.
\end{enumerate}
\end{theorem}
\begin{proof} Let $\alg A \in K(\vv V)$ be finitely generated, say by $\vuc gn$; then $\alg A^-$ is finitely generated by $g_1 \meet 1,\dots ,g_n \meet 1,\sim g_1\meet 1,\dots,\sim g_n\meet 1$ and it is finite, since it belongs lo the locally finite variety $\vv V$. Hence $K(\alg A^-)$ is finite and so is $\alg A \le K(\alg A^-)$.
Conversely, suppose $\alg A \in \vv V$ is finitely generated by $\vuc an$; then  consider $\alg B$ the smallest admissible subalgebra of $K(\alg A)$, which is finitely generated by $(a_1,1),\dots,(a_n,1)$ and hence it is finite as $\alg B \in K(\vv V)$. So $\alg A \cong \alg B^-$ is finite.

For 2., suppose that $\vv V$ has the FEP and let $\alg A \in K(\vv V)$ and $\alg B$ be a finite partial subalgebra of $\alg A$.  Let $C = \{u \meet 1: u \in B \cup \nneg B\}$; then $\alg C$ is a finite partial subalgebra of $\alg A^-$ and hence there is an embedding $f:\alg C \longrightarrow \alg D$, with $\alg D$ a finite algebra in $\vv V$.
Let now $C^\# = \{(b \meet 1, \nneg b \meet 1): b \in B\}$; then $\alg C^\#$ is a partial subalgebra of $K(\alg A^-)$ and  $g: b \longmapsto (b\meet 1, \nneg b \meet 1)$ is an embedding of $\alg B$ into $\alg C^\#$ (because of Lemma \ref{inclusions}). Finally $f^\#: (b \meet 1, \nneg b \meet 1) \longmapsto (f(b \meet 1), f(\nneg b \meet 1))$ is an embedding of $\alg C^\#$ in $K(\alg D)$. Therefore the composition $f^\#\circ g$ is an embedding of $\alg B$ into $K(\alg D)$; and since the latter is finite $K(\vv V)$ has the FEP. Conversely suppose $K(\vv V)$ has the FEP, let $\alg A \in \vv V$ and let $\alg B$ be a finite partial subalgebra of $\alg A$. Then we can consider $B \times B$ as the universe of a partial subalgebra of $K(\alg A)$; with a slight abuse of language we call it $K(\alg B)$ and we observe that $K(\alg B)^- = \alg B$. Then $K(\alg B)$ is embeddable in a finite $\alg D \in K(\vv V)$ and so $\alg B$ is embeddable in $\alg D^-$; since $\alg D^-$ is finite and belongs to $\vv V$ the conclusion holds.

Finally the proof of  3. is straightforward argument using Lemmas \ref{subdir} and \ref{Kproperties}.
\end{proof}

 If $\alg L$ is any lattice a pair $(a,b) \in L^2$ of elements of $L$ is a {\bf splitting pair} if $L$ is equal to the disjoint union of the ideal generated by $a$ and the filter generated by $b$.  If $\vv V$ is any variety, an algebra $\alg A \in \vv V$ is {\bf splitting in $\vv V$} if $\VV (\alg A)$ is the right member of a splitting pair in the lattice of subvarieties of $\vv V$. A more transparent definition is the following: $\alg A$ is splitting in $\vv V$ if there is a subvariety $\vv W_\alg A \sse \vv V$ (the {\bf conjugate variety of $\alg A$}) such that for any variety $\vv U \sse \vv V$ either $\vv U\sse \vv W_\alg A$ or $\alg A \in \vv U$.

Splitting lattices were introduced by P. Whitman \cite{Whitman1943} in the early 40' s; thirty years later R. McKenzie  \cite{McKenzie1972} explored the concept with greater fortune. Note that if $(a,b)$ is a splitting pair then, by a standard lattice theoretic argument, $b$ must be completely meet irreducible and $a$ completely join irreducible. If we apply this to the splitting pair $(\vv W_\alg A,\VV(\alg A))$ then $\vv W_\alg A$ is axiomatized by a single equation and $\VV(\alg A)$ is generated by a finitely generated subdirectly irreducible algebra. If moreover $\vv V$ is generated by its finite algebras, then $\alg A$ must in fact be finite and if $\vv V$ is also congruence distributive then (by J\'onsson Lemma) $\alg A$ is  unique.  We summarize all this in a theorem:

\begin{theorem}\label{mckenzie} \cite{McKenzie1972} Let $\vv V$ be a congruence distributive variety that is generated by its finite algebras; then any splitting algebra $\alg A$ in $\vv V$ is finite and subdirectly irreducible. Moreover the conjugate variety $\vv W_\alg A$ is axiomatizable by a single equation in the language of $\vv V$.
\end{theorem}

 The equation whose existence is postulated by Theorem \ref{mckenzie} is called the {\bf splitting equation} of $\alg A$.
Note that the theorem does not say that if the hypotheses hold, then there is a splitting algebra in $\vv V$; nor its proof produces an effective way of determining the splitting equation of $\alg A$, in case it is splitting. Both the existence and the splitting equation require {\em ad hoc} arguments; and  in \cite{McKenzie1972} R. McKenzie did exactly that, characterizing  the splitting algebras in the variety of all lattices and also giving an algorithm (still the only one available) for their splitting equations. The seminal result here is the one in \cite{KowalskiOno2000}: the only splitting algebra in $\mathsf{BCIRL}$ is the two element Boolean algebra $\mathbf 2$. The problem whether $\mathbf 2$ is splitting in $\mathsf{CIRL}$ is open, while it is a consequence of a general result by T. Kowalski (see Exercise 16, Chap. 10 in \cite{GJKO}) that $\alg K_3$ is not splitting in $\mathsf{CRL}$ (but we do not know if it is splitting in $\mathsf{KL}$). Splitting algebras in divisible subvarieties of $\mathsf{CIRL}$ and $\mathsf{BCIRL}$ have been completely characterized in  \cite{Agliano2017c}, \cite{Agliano2018a} and \cite{Agliano2018b}, while splittings in non divisible subvarieties of $\mathsf{BCIRL}$ have been  considered in  \cite{AglianoUgolini2019a}.

What are the relations between the splitting algebras in $\vv V$ and the splitting algebras in $K(\vv V)$? There does not seem to be any direct relation, though
at least we have a necessary condition (Theorem \ref{kalmansplit} below) and a knowledge of the splitting algebras in $\vv V$ can give a suggestion of what the splitting algebras in $L(\vv V)$ may be.

\begin{theorem}\label{kalmansplit} Let $\vv V$ be a subvariety of $\mathsf{CIRL}$ and let $\alg A \in \vv V$.
If $K(\alg A)$ is splitting in $K(\vv V)$ with conjugate variety $K(\vv W)$, then $\alg A$ is splitting in $\vv V$ with conjugate variety $\vv W$.
\end{theorem}
\begin{proof} Obvious from Lemma \ref{eqn} and Theorem \ref{axioms}.
\end{proof}

\subsection{Nelson paraconsistent lattices}

By Theorem \ref{kalmansplit} the lattice $\Lambda(\mathsf{NPcL})$ contains a copy of $\Lambda(\mathsf{Br})$, the lattice of subvarieties of Brouwerian lattices. Since we know that the latter is already hopelessly complex we will make no attempt to investigate it; however it is clear from the results in Section \ref{section3} that $\mathsf{PKL}$ is the only atom in $\Lambda(\mathsf{NPcL})$ and that $K(\mathsf{GBA}) = \VV(\alg K_4)$ is the only almost minimal variety above $\mathsf{PKL}$.

On the other hand splitting algebras in $\mathsf{NPcL}$ are easily classified due to a more general result. A variety  $\vv V$ has {\bf equationally definable principal congruences} (EDPC for short) if there are quaternary  terms $\sigma_i, \tau_i$ of $\vv V$,  $i=1,\dots,n$, such that for any  $\alg A \in \vv V$ and $a,b,c,d \in A$, $(c,d) \in \op{Cg}_\alg A(a,b)$ if and only if
$$
\sigma_i(a,b,c,d) = \tau_i(a,b,c,d)\quad i=1,\dots,n.
$$
Varieties with EDPC have been thoroughly investigated and the fact we are interested in is:  in a variety with EDPC any {\em finitely presentable}\footnote{i.e.  a homomorphic image of a finitely generated free algebra in the variety.} subdirectly irreducible algebra is splitting \cite{EDPC1}. Subvarieties of $\mathsf{CRL}$ with EDPC have been completely classified (see Theorem 3.2 in \cite{Agliano2017c},  as well as the discussion preceding it): a subvariety  of $\mathsf{CRL}$ has EDPC if and only if it is {\em $n$-subcontractive} for some $n$, i.e. it satisfies the equation $(x \meet 1)^n \app (x \meet 1)^{n+1}$.
Since Brouwerian lattices are idempotent, by Theorem \ref{axioms}, $\mathsf{NPcL}$  is $1$-subcontractive and thus has EDPC. Moreover since Brouwerian lattices have the
FEP, so does $\mathsf{NPcL}$, hence any splitting algebra in $\mathsf{NPcL}$ must be finite.  Putting everything together we conclude that the splitting algebras in $\mathsf{NPcL}$ are exactly the finite subdirectly irreducible algebras.  Note that the same conclusion can be reached for $\mathsf{GNPcL}$: $1$-subcontractivity is obvious, while the FEP comes from the fact that G\"odel hoops are locally finite.

The lattice $\Lambda(\mathsf{GNPcL})$ is of course much simpler and we proceed to describe it.
Let's first examine $K(\alg G_3)$; this algebra is  subdirectly irreducible and belongs to $\mathsf{GNPcL}$. We have that
\begin{align*}
\SU\HH(K(\alg G_3)) &= K(\SU\HH(\alg G_3)) = K(\SU(\alg G_3) \cup \SU(\mathbf 2)) \\
&= S(K(\alg G_3)) \cup S(K(\mathbf 2))
\end{align*}
so it follows that the only possible new contributions can come from $K(\alg G_3)$.
By Theorem \ref{kalmanheyting} there is only one proper subalgebra of $K(\alg G_3)$ whose negative cone is isomorphic with $\alg G_3$;   we call it $\alg K_8$, its universe  is $K(\alg G_3) \setminus \{(0,0)\}$, it is generated by two elements and its structure is described in Figure \ref{K8}.

\begin{figure}[htbp]
\begin{center}
\begin{tikzpicture}[scale=0.7]
\draw (1,5) -- (0,6) -- (-1,5);
\draw (-1,3) --(0,2) -- (1,3)-- (2,4)
-- (1,5) -- (0,4) -- (-1,3);
\draw (1,3) -- (-1,5);
\draw[fill] (1,3) circle [radius=0.05];
\draw[fill] (2,4) circle [radius=0.05];
\draw[fill] (1,5) circle [radius=0.05];
\draw[fill] (0,4) circle [radius=0.05];
\draw[fill] (-1,3) circle [radius=0.05];
\draw[fill] (0,2) circle [radius=0.05];
\draw[fill] (1,3) circle [radius=0.05];
\draw[fill] (-1,5) circle [radius=0.05];
\draw[fill] (0,6) circle [radius=0.05];
\node[right] at (1,3) {\footnotesize $x=x(x \join y)$};
\node[left] at (-1,3) {\footnotesize $y$};
\node[right] at (2,4) {\footnotesize $1$};
\node[left]  at (-1,5) {\footnotesize $\nneg y$};
\node[right] at (1,5) {\footnotesize $\nneg x$};
\end{tikzpicture}
\end{center}
\caption{The structure of $\alg K_8$\label{K8}}
\end{figure}
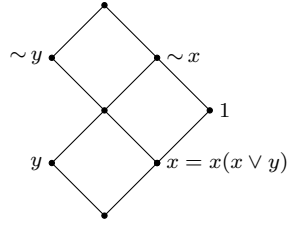

However it is easy to check that the set $\{x,\nneg x, x \join y, 1\}$ is the universe of a subalgebra of $\alg K_8$, isomorphic with $\alg K_4$ and that the only subalgebras of $K(\alg G_3)$ whose negative cone is not $\alg G_3$ are again isomorphic with $\alg K_3$ and $\alg K_4$. Hence $\VV(\alg K_3) \le \VV(\alg K_4) \le \VV(K(\alg G_3))$.   This argument can be generalized: let  $\alg K_{n^2}= K(\alg G_n)$ and let $K_{n^2-1}$ the subalgebra of $\alg K_{n^2}$ whose universe is
$K_{n^2}\setminus \{0,0\}$.

\begin{theorem}\label{kghsubalgebras} The nontrivial subalgebras (up to isomorphism) of $\alg K_{n^2}$ are the algebras $\alg K_{m^2}$ and $\alg K_{m^2-1}$ for $m=2,\dots,n$.
\end{theorem}
\begin{proof} We induct on $n$; the case $n=2$ has been discussed above. Suppose then the statement to be true for $K(\alg G_n)$; now any any subset of $\alg G_{n+1}$ that has $n$ elements and contain $1$, is a subalgebra of $\alg G_{n+1}$ isomorphic with $\alg G_n$. It follows that $\alg K_{n^2} \le \alg K_{(n+1)^2}$, so we have only to show that the only nontrivial new subalgebra of $\alg K_{(n+1)^2}$ is $\alg K_{(n+1)^2-1}$.  But a ``new'' subalgebra $\alg A$ of $\alg K_{(n+1)^2-1}$ must be such as
$\alg A^- = \alg G_{n+1}$; and by Theorem \ref{admissibleordinalsum} (or Lemma \ref{kalmanheyting}) there is only one of them, that is exactly $\alg K_{(n+1)^2-1}$. This concludes the proof.
\end{proof}

Let $\vv W \sse K(\vv V)$, where $\vv V$ is any subvariety of $\mathsf{CIRL}$; then $\vv W \sse K(\vv W^-)$, i.e. any subvariety of $\vv K(\vv V)$ is contained in a Kalman subvariety. Knowing the shape of $\Lambda(\mathsf{GH})$ and from the fact that $K$ is a lattice embedding we get at once:

\begin{corollary} The lattice $\Lambda(\mathsf{GNPcL})$ is a chain of type $\o+1$, whose proper subvarieties are $V(\alg K_{n^2})$ and $V(\alg K_{n^2-1})$ for $n\ge 2$.
Moreover $\VV(\alg K_u) \le \VV(\alg K_v)$ if and only if $u \le v$ in $\mathbb N$.
\end{corollary}

\begin{figure}[htbp]
\begin{center}
\begin{tikzpicture}[scale=.6]
\draw (0,2) -- (0,4);
\draw[dashed] (0,4) --(0,6);
\draw (0,6) -- (0,7) ;
\draw[dashed] (0,7) --(0,9);
\draw[fill] (0,2) circle [radius=0.05];
\draw[fill] (0,3) circle [radius=0.05];
\draw[fill] (0,4) circle [radius=0.05];
\draw[fill] (0,6) circle [radius=0.05];
\draw[fill] (0,7) circle [radius=0.05];
\draw[fill] (0,9) circle [radius=0.05];
\node[right] at (0,2) {\footnotesize $\mathsf{T}$};
\node[right] at (0,3) {\footnotesize $\mathsf{\VV(\alg K_3)}$};
\node[right] at (0,4) {\footnotesize $\VV(\alg K_4)=K(\VV(\alg G_2))$};
\node[right] at (0,6) {\footnotesize $\VV(\alg K_{n^2-1})$};
\node[right] at (0,7) {\footnotesize $\VV(\alg K_{n^2}) = K(\VV(\alg G_n))$};
\node[right] at (0,9) {\footnotesize $\mathsf{GNPcL} = K(\mathsf{GH})$};
\end{tikzpicture}
\end{center}
\caption{$\Lambda(\mathsf{GNPcL})$\label{}}
\end{figure}

Hence any proper subvariety of $\Lambda(\mathsf{NPcL})$ is finitely generated either by $\alg K_{n^2}$ or by $\alg K_{n^2-1}$ for $n \ge 2$, and we know that those varieties are all distinct by J\'onsson Lemma. Therefore if $\VV(\alg K_v)$ is a cover of $\VV(\alg K_u)$ then any equation holding in $\alg K_u$ but not in $\alg K_v$
is a splitting equation for $\alg K_v$ (and the conjugate variety is $\VV(\alg K_u)$). And we are sure that at least one of such equations must exist, since $\alg K_u$ and $\alg K_v$ generate distinct varieties. Finding the equation is a different story though, but  we can be more specific in at least one case. Suppose that we want to find the splitting equation
for $\VV(\alg K_{(n+1)^2 -1})$, with conjugate variety $\VV(\alg K_n)$.  A general result about hoops is very useful; we state it as a lemma that is the combination of Theorem 3.7 and Lemma 4.2 in \cite{AglianoMontagna2003}:

\begin{lemma} Let $\alg A$ be a totally ordered hoop and let $\bigoplus_{i\in I} \alg A_i$ a decomposition of $\alg A$ into sum irreducible hoops. Then
\begin{enumerate}
\item for all $i \in I$, $\alg A_i$ is a Wajsberg hoop;
\item $|I| \le n$  if and only if $\alg A$ satisfies the equation
\begin{equation}
\bigwedge_{i=0}^{n-1} ((x_{i+1} \imp x_1)\imp x_i) \le \bigvee_{i=1}^n x_i \tag{$\lambda_n$}
\end{equation}
\end{enumerate}
\end{lemma}

Since $\alg G_n = \bigoplus_{i=1}^n \mathbf 2$ and $\alg G_n = \alg K^-_{n^2} = \alg K^-_{n^2-1}$ we get at once that $\alg K_n \vDash \kappa(\lambda_n)$ and
$\alg K_{(n+1)^2-1} \not \vDash \kappa (\lambda_n)$; hence $\kappa(\lambda_n)$ is a splitting equation for $\alg K_{(n+1)^2-1}$.

\subsection{Product K-lattices}
 A basic hoop is a {\bf product hoop} if it satisfies the equation
$$
(y \imp z) \join ((x \imp xy) \imp y) \app 1.
$$
The variety $\mathsf{PH}$ of product hoops has been studied in \cite{AFM}. From there we recall:
\begin{enumerate}
\ib the subdirectly irreducible product hoops are exactly $\mathbf 2$ and $\mathbf 2 \oplus \alg C$ where $\alg C$ is a totally ordered cancellative hoop;
\ib $\HH\SU\PP_u(\mathbf 2 \oplus \alg C_\o) = \HH\SU\PP_u(\mathbf 2) \cup (\SU\PP_u(\mathbf 2) \oplus \alg C_\o) = \{\mathbf 2\} \cup \{\mathbf 2 \oplus \alg C: \alg C \in \vv C\}$, hence if $\alg C \in \vv C$ is nontrivial, $\VV(\mathbf 2 \oplus \alg C) = \mathsf{PH}$;
\ib $\mathsf{CPH}= \VV(\mathsf 2, \alg C_\o)$  is axiomatized modulo product hoops by Tanaka's equation (T), hence it is the coatom in $\Lambda(\mathsf{PH})$;
\ib the lattice of subvarieties of product hoops has exactly five elements; the three proper nontrivial ones are the variety $\mathsf{GBA}$ of generalized boolean algebras, the variety $\vv C$ of cancellative hoops and the variety $\mathsf{CPH}$.
\end{enumerate}

\begin{lemma}\label{phsubirr} The subdirectly irreducible algebras in $K(\mathsf{PH})$ are exactly $\alg K_3$, $\alg K_4$, $K(\alg C)$ for any totally ordered cancellative hoop $\alg C$ and $K(\mathbf 2 \oplus \alg C)$ for any totally ordered cancellative hoop $\alg C$.
\end{lemma}
\begin{proof} Since $\mathsf{PH} = \VV(\mathbf 2 \oplus \alg C_\o)$   we get at once that $K(\mathsf{PH}) = \VV(K(\mathbf 2 \oplus \alg C_\o))$. Hence by J\'onsson Lemma all the subdirectly irreducible are in
\begin{align*}
\HH\SU\PP_u(K(\mathbf 2 \oplus \alg C_\o)) &\sse \SU K(\HH\SU\PP_u(\mathbf 2 \oplus \alg C_\o)) \\
&= \{\alg K_4, \alg K_3, K(\alg C), K(\mathbf 2 \oplus \alg C): \alg C \in \mathsf{C}\ \text{and totally ordered}\}
\end{align*}
Since all these algebras are subdirectly irreducible the conclusion holds.
\end{proof}

To determine  $\Lambda(K(\mathsf{PH}))$ we need a simple lemma first.

\begin{lemma}\label{joinirr} Let $\alg A \in \mathsf{KL}$; then $1$ is join irreducible in $\alg A^-$ if an only if it is join irreducible in $\alg A$.
\end{lemma}
\begin{proof} Since $\alg A^-$ is a sublattice of $\alg A$ one implication is obvious,so we assume that $1$ is  join irreducible in $\alg A^-$ .
We can represent $\alg A$ as a subalgebra of $K(\alg A^-)$ and hence the elements of $\alg A$ are pairs $(a,b)$ with $a,b \in A^-$. So suppose that there are
$(a,b),(c,d) \in A$ with $(a,b) \join (c,d) = (1,1)$; then
$$
(1,1) = (a,b) \join (c,d) = (a \join b, c \meet d).
$$
Since $\alg A^-$ is integral we must have $c=d=1$ and, since $1$ is join irreducible in $\alg A^-$, either $a=1$ or $b=1$. It follows that either $(a,b) =(1,1)$ or $(c,d) =(1,1)$ and so $1$ is join irreducible in $\alg A$.
\end{proof}

Now if we set $\alg K_\o = K(\alg C_\o)$ we have:

\begin{theorem}\label{productlattice} $\Lambda(K(\mathsf{PH}))$ is the lattice in Figure \ref{plattice}. In details:
\begin{enumerate}
\item $K(\mathsf{CPH})$ is the only coatom in the lattice;
\item the splitting algebras are:
\begin{enumerate}
\item  $\alg K_3$, with splitting equation $x \imp_1 xy\app y \meet 1$  and   conjugate variety $K(\mathsf{C})$;
\item  $\alg K_4$, with splitting equation
$$
((x \imp_1 xy) \imp (y \meet 1)) \join ((x \meet \nneg x) \imp_1 (y \join \nneg y)) \app 1
$$
and conjugate variety $\VV( \alg K_3,\alg K_\o)$;
\item $\alg K_\o$, with splitting equation $(x\meet 1)^2 \app x \meet 1$ and conjugate variety $K(\mathsf{GBA})$;
\item $K(\mathbf 2 \oplus \alg C)$, where $\alg C$ is a totally ordered cancellative hoop; in this case the splitting equation is
 $$
(x \imp_1 y) \imp (y \meet 1) \app (y \imp_1 x) \imp (x \meet 1)
$$
and the conjugate variety  is $K(\mathsf{CPH})$.
\end{enumerate}
\end{enumerate}
\end{theorem}
\begin{proof} By Lemma \ref{phsubirr} the only thing we have to prove is that if $\alg C$ is a totally ordered cancellative hoop, then $\VV(K(\alg C)) = \VV(\alg K_\o)$. But this is obvious, since the analogous statement holds for cancellative hoops. Let then $\vv V$ be any proper subvariety of $K(\mathsf{PH})$; then $\vv V$ cannot contain any algebra as $K(\mathbf 2 \oplus \alg C)$ where $\alg C$ is totally ordered and cancellative. Otherwise
$$
\vv V \supseteq \VV(K(\mathbf 2 \oplus \alg C)) = K(\VV(\mathbf 2 \oplus \alg C)) = K(\mathsf{PH})
$$
that contradicts $\vv V \subsetneq K(\mathsf{PH})$. So the only subdirectly irreducible algebras in  $\vv V$ can be $\alg K_3,\alg K_4, \alg K(\mathsf C)$; since
$K(\mathsf{CPH}) = \VV(\alg K_4,\alg K_\o)$, $\vv V \sse K(\mathsf{CPH})$ and 1. follows.

For (2), the fact that the algebras are splitting with the desired conjugate variety follows by inspection of the lattice.  The splitting equations are also obvious except in case (b); it is clear that a splitting equation for $\alg K_4$ is any equation holding in $\alg K_3,\alg K_\o$ but not in $\alg K_4$. We also observe that the equation
$$
(x \meet \nneg x) \imp_1 (y \join \nneg y) \app 1
$$
is equivalent to $x \meet \nneg x \le y \join \nneg y$ and this clearly holds in $\alg K_3$ but not in $\alg K_4$. Our final observation is that $1$ is join irreducible in both $\alg K_3$ and $\alg K_\o$ by Lemma \ref{joinirr}. Given this it is a simple  exercise to check that the equation in (b) has the desired properties.
\end{proof}

\begin{figure}[htbp]
\begin{center}
\begin{tikzpicture}
\draw (0,2) -- (1,3) -- (0,4) -- (-1,5) -- (-1,6) -- (-1,5) -- (-2,4) -- (-1,3) --(0,4) -- (-1,3) -- (0,2);
\draw[fill] (0,2) circle [radius=0.05];
\draw[fill] (1,3) circle [radius=0.05];
\draw[fill] (0,4) circle [radius=0.05];
\draw[fill] (-1,5) circle [radius=0.05];
\draw[fill] (-1,6) circle [radius=0.05];
\draw[fill] (-2,4) circle [radius=0.05];
\draw[fill] (-1,3) circle [radius=0.05];
\node[right] at (0,2) {\footnotesize $\mathsf{T}$};
\node[right] at (1,3) {\footnotesize $\mathsf{K(\mathsf{C}) =\VV(\alg K_\o)}$};
\node[right] at (0,4) {\footnotesize $\VV(\alg K_3,\alg K_\o)$};
\node[right] at (-1,5) {\footnotesize $K(\mathsf{CPH})$};
\node[right] at (-1,6) {\footnotesize $K(\mathsf{PH})$};
\node[left] at (-2,4) {\footnotesize $K(\mathsf{GBA}) = \VV(\alg K_4)$};
\node[left] at (-1,3) {\footnotesize $\mathsf{PKL} =\VV(\alg K_3)$};
\end{tikzpicture}
\end{center}
\caption{$\Lambda(K(\mathsf{PH}))$\label{plattice}}
\end{figure}

\subsection{Wajsberg K-lattices}\label{wklattices}

A K-lattice $\alg A$ is a {\bf Wajsberg K-lattice} if $\alg A^- \in \mathsf{WH}$. It is clear from the results of the previous sections that  Wajsberg K-lattices form a variety $\mathsf{WKL}$ and that $\mathsf{WKL} =K(\mathsf{WH})$, i.e. it is a Kalman variety. The lattice $\Lambda(\mathsf{WH})$ has been investigated in \cite{AglianoPanti1999}; its shape is very complex but it has some nice features that  we proceed to describe. First we have to introduce a class of infinite bounded Wajsberg chains. If $\alg A,\alg B$ are commutative $\ell$-groups we define the {\bf lexicographic product} $\alg A\times^l\alg B$ as the direct product of the two groups with the lexicographic ordering; it is clearly a commutative $\ell$-group.  Given any commutative $\ell$-group $\alg A$ with a strong unit $u$ we define $\Gamma(\alg A,u)$ \cite{Mun1986} as the commutative residuated integral lattice whose universe is the interval $[0,u]$ and the operations are ($+$ and $-$ are the group operations)
$$
a \imp b = (u -a+b) \meet u\qquad ab= (a +b -u) \join 0.
$$
It is easy to check that $\Gamma(\alg A,u)$ is a Wajsberg hoop and that $\Gamma(\mathbb Z,n) \cong \alg \L_n$; we define
$$
\alg \L^\o_n = \Gamma(\mathbb Z \times^l\mathbb Z, (n,0)).
$$
We have:
\begin{enumerate}
\ib the only atoms in $\Lambda(\mathsf{WH})$ are $\mathsf C$ and $\mathsf{GBA}$;
\ib the only splitting algebra in $\Lambda(\mathsf{WH})$ is $\mathbf 2$ \cite{Agliano2017c};
\ib  $\alg \L_n, \alg \L^\o_n \le \alg \L^\o_m$  if and only if $n\mathrel{|} m$;
\ib  the almost minimal varieties in $\Lambda(\mathsf{WH})$ are $\VV(\alg \L_p)$ and $\VV(\alg \L_p) \join \vv C$ for $p$ prime;
\ib each proper subvariety has only finitely many subvarieties (we say sometimes that has {\em finite height}): if $\vv V$ is proper than there is a finite subset $X$ of $\{\mathsf C,\alg \L_n: n \in \mathbb N\}$ and a finite subset $Y\sse\{\alg \L^\o_m: m \in \mathbb N\}$ with $\vv V = \VV(X \cup Y)$ \cite{AglianoPanti1999};
\ib a subvariety $\vv V$ of $\mathsf{WH}$ is locally finite if and only if it is finitely generated \cite{Agliano2015}.
\end{enumerate}

Using the information in the previous sections we can derive:

\begin{theorem} The only atoms in $\Lambda(\mathsf{WKL})$ are $\mathsf{PKL}$ and $K(\mathsf C)$; the almost minimal varieties are
\begin{enumerate}
\item $K(\mathsf{GBA})$ and $\VV(\alg K_{0,p})$ for $p$ prime, $p>2$;
\item $K(\vv C) \join \vv V$ where $\vv V$ is any variety in 1.
\end{enumerate}
Moreover the varieties in 1. are the only almost minimal varieties above $\mathsf{PKL}$.
\end{theorem}
\begin{proof} Let $\vv V$ be any subvariety of $K(\mathsf{WH})$ and let $\alg A\in \vv V$ be subdirectly irreducible; then $\alg A^-$ is subdirectly irreducible as well, so it is a totally ordered Wajsberg hoop. But a totally ordered Wajsberg hoop is either bounded or cancellative (see for instance \cite{BlokFerr1993}); if $\alg A^-$ is bounded, than $\alg K_3 \le \alg A$ and if $\alg A^-$ is cancellative, then $K(\mathsf{C}) \sse \vv V$. This proves the first claim.

Suppose that $\vv V$ is an almost minimal variety in $\Lambda(\mathsf{WKL})$ ; then $\vv V^-$ is a subvariety of $\mathsf{WH}$  that is an atom in $\Lambda(\mathsf{WH})$. In fact otherwise there would be an atom $\vv A$ of $\mathsf{CIRL}$ such that $\vv A \subsetneq \vv V$, and since $K(\vv V^-)$ is the smallest Kalman variety containing $\vv V$ (Lemma \ref{techlemma}(4)), we would have $\mathsf{PKL} \subsetneq K(\vv A) \subsetneq \vv V \sse K(\vv V^-)$ (remember, $\mathsf{PKL}$ is not a Kalman variety).
The rest of the statement follows form the description of the atoms in $\Lambda(\mathsf{WH})$ above and the description of the subalgebras of $K(\alg \L_p)$ in Section \ref{almostmin}.
\end{proof}

Also the splitting algebras in $K(\mathsf{WH})$ are simple enough to determine; Wajsberg hoops have the FEP \cite{BlokFerr2000}, so $\mathsf{WKL}$ has the FEP as well (Theorem \ref{FEP}). It follows that any splitting algebra in $\mathsf{WKL}$ must be finite. Moreover it well known that, if $L\sse \mathbb N$ is infinite, then
$\VV(\{\alg \L_l : l \in L\}) = \mathsf{WH}$.

\begin{theorem} If $\alg A$ is splitting in $\mathsf{WKL}$ then $\alg A^- \cong \mathbf 2$; in particular  $\alg K_3$ is splitting in $\mathsf{WKL}$.
\end{theorem}
\begin{proof} The first part is immediate from Theorem \ref{kalmansplit} and the fact that $\mathbf 2$ is the only splitting algebra in $\mathsf{WH}$.

Let now $\alg A$  be a subdirectly irreducible algebra such that $\alg K_3 \notin \VV(\alg A)$. Then $\alg A^-$ cannot be bounded, hence, since it is a totally ordered Wajsberg hoop, it must be cancellative. It follows that $\VV(\alg A^-) = \vv C$ and hence $K(\VV(\alg A^-)) = \VV(K(\alg A^-)) = K(\mathsf{C})$ (since $K(\mathsf{C})$ is almost minimal and $\alg K_3 \notin \VV(\alg A)$). It follows that $\alg K_3$ is splitting with conjugate variety $K(\vv C)$.
\end{proof}

We were not able to determine whether $\alg K_4$ is splitting in $\mathsf{WKL}$ so we leave it as an open problem.

To progress further in our investigation we need to use the results in \cite{AglianoPanti1999} and \cite{AglianoMontagna2003} about totally ordered Wajsberg hoops. If $\alg A$ is totally ordered Wajsberg hoop we define the {\bf radical} of $\alg A$ as $\op{Rad}(\alg A)=\{a: o(a) =\o\}$; it is easy to check that $\op{Rad}(\alg A)$ is a filter and a cancellative subhoop of $\alg A$. Next we define the {\bf order} and the {\bf rank} of $\alg A$ as
$$
o(\alg A) = \sup\{o(a): a \in A\}\qquad\qquad r(\alg A) = o(\alg A/\op{Rad}(\alg A))
$$
and we observe that $r(\alg A) = \o$ implies $o(\alg A) = \o$. Note that $o(\alg\L_n)=r(\alg\L_n) =n$, $o(\alg \L^\o_n) = \o$ and $r(\alg \L^\o_n) = n$.

\begin{lemma}\label{wajsberg} Let $\alg A$ be a totally ordered Wajsberg hoop. Then:
\begin{enumerate}
\item if $r(\alg A)$ is finite and $\alg A$ is not simple, then $\HH\SU\PP_u(\alg A)$ consists of all totally ordered cancellative hoops and of all the totally ordered Wajsberg hoops whose rank divides $r(\alg A)$ (\cite{AglianoMontagna2003}, Theorem 7.9);
\item if $r(\alg A)=n$ is finite and $\alg A$ is simple then $\alg A \cong \alg \L_n$ and $\HH\SU\PP_u(\alg A)$ consists of all finite Wajsberg chains whose order divides $n$ (\cite{AglianoMontagna2003}, Theorem 7.9);
\item if $o(\alg A) = \o$ and $r(\alg A) = n$, then $\VV(\alg A) = \VV(\alg \L_n^\o)$ (\cite{AglianoPanti1999}, Theorem 2.4);
\item if $o(\alg A) = n$, then $\VV(\alg A) = \VV(\alg \L_n)$ (\cite{AglianoPanti1999}, Theorem 2.4).
\end{enumerate}
\end{lemma}

Now we are able to describe $\Lambda(\mathsf{WKL})$ to a certain extent. Let's say that a subvariety $\vv W \sse \mathsf{WKL}$ is {\bf small} if $\vv W^-$ is a proper subvariety of $\mathsf{WH}$.

Recalling Theorem \ref{kalmanwajsberg} and the fact that all bounded Wajsberg hoops are involutive, all the admissible subalgebras of $\alg K_{n,n}=K(\alg\L_n)$ and $\alg K^\o_n = K(\alg \L^\o_n)$ are in one to one correspondence with the lattice filters of $\alg\L_n$ and $\alg\L^\o_n$, respectively, and as these algebras are chains they are therefore totally ordered with bottom and top elements given by the trivial filters. Moreover, we have that
\begin{itemize}
	\item $K(\alg \L_n,F)\leq K(\alg \L_m,G)$ if and only if $n|m$ and $\langle F\rangle_m\subseteq G$, where $\langle F\rangle_m$ is the lattice filter generated by $F\subseteq \alg \L_n$, viewed as a subset of $\alg \L_m$. This can be simplified as each lattice filter in $\alg\L_n$ is principal, and therefore if $F=\langle x^r\rangle_n$ and $G=\langle x^s\rangle_m$, then $\alg K_{r,n}=K(\alg\L_n,F)\leq K(\alg\L_m,G)=\alg K_{s,n}$ if and only if there is a $k \in \mathbb N$ with $nk=m$ and $s \ge rk$.
	\item $K(\alg\L_n,F)\leq K(\alg\L^\o_m,G)$ if and only if $n|m$ and $\langle F\rangle_m\subseteq G$.
	\item $K(\alg\L^\o_n,F)\leq K(\alg\L^\o_m,G)$ if and only if $n|m$ and $\langle F\rangle_m\subseteq G$.
\end{itemize}

To simplify the notation we set $\alg K^\o_{n,F} = K(\alg \L^\o_n,F)$ for $n\ge 1$, and recall that $\alg K_{0,1} = \alg K_3$, $\alg K_{1,1} = \alg K_4$.

Observe that any subvariety of $\mathsf{WKL}$ containing some $\alg K^\o_{n,F}$ with $F$ a lattice filter containing the \textsl{top component} of $\alg \L^\o_n$ will have infinite height. However, we do have the following result.

\begin{theorem}\label{smallsub} Each small subvariety $\vv W \sse \mathsf{WKL}$ has one of the following three forms
\begin{enumerate}
\item $\VV(\alg K_{r_1,n_1}, \dots, \alg K_{r_k,n_k})$, $k\ge 1$, $r_i \le n_i$, $i=1,\dots, k$;
\item $\VV(\alg K_{r_1,n_1}, \dots, \alg K_{r_k,n_k}, \alg K_\o)$, $r_i \le n_i$, $i=1,\dots, k$
\item $\VV(\alg K_{r_1,n_1}, \dots, \alg K_{r_k,n_k}, \alg K^\o_{m_1,F_1},\dots, \alg K^{\o}_{m_s,F_s})$, $k,s\ge 1$, $r_i \le n_i$, $i=1,\dots, k$.
\end{enumerate}
\end{theorem}
\begin{proof} Since any small subvariety is contained in a proper Kalman subvariety of $\mathsf{WKL}$, it is enough to prove the first statement for $K(\vv V)$, where $\vv V$ is a proper subvariety of $\mathsf{WH}$. By Theorem 2.5 in \cite{AglianoPanti1999} there are only three possibilities, that we will examine one by one.
In the first case $\vv V = \VV(\alg \L_{n_1}, \dots, \alg \L_{n_k})$; by what we said above $\HH\SU\PP_u(\alg \L_{n_1},\dots,\alg \L_{n_k}) = \{\alg \L_m: m \mathrel{|} n_i\ \text{for some $i \le k$}\}$. The conclusion then follows from Lemma \ref{techlemma} and the description of the subalgebras of $K(\alg \L_n)$ in Section \ref{almostmin}. In the second case $\vv V = \VV(\alg \L_{n_1}, \dots, \alg \L_{n_k},\alg C_\o)$; then the conclusion follows from first and the fact that $\VV(K(\alg A)) = \VV(K(\alg C_\o))$ whenever $\alg A$ is totally ordered and cancellative.

In the third case $\vv V = \VV(\alg \L_{n_1}, \dots, \alg \L_{n_k},\alg \L^\o_{m_1}, \dots, \alg \L^\o_{m_s})$ and the conclusion follows from the facts that $K(\alg C_\o) \le K(\alg \L_n^\o,F)$ for all $n$ and all lattice filters $F$, and the observations previous to this Theorem.
\end{proof}

Note that Theorem \ref{smallsub} implies that the join of two small subvarieties is again small, hence the small subvarieties form a sublattice of $\Lambda(\mathsf{WKL})$; observe also that this sublattice cannot have a top. In fact if $\vv V, \vv W \sse \mathsf{WH}$, then $\vv V \sse \vv W$ if and only $K(\vv V) \sse K(\vv W)$; so if there were a largest small subvariety, there would also be a largest proper subvariety of $\mathsf{WH}$, and that is known to be false.

We will finish Section \ref{wklattices} by studying the inclusion relations among subvarieties of the first two forms according to Theorem \ref{smallsub}, as in this case we can guarantee that these subvarieties will be of finithe height.

We introduce a set-theoretic representation of the sublattice of small subvarieties.  Let $\mathbb N^\nabla = \{(n,m) \in \mathbb N^2: n \le m\}$;
we define a relation $\ll$ on $\mathbb N^\nabla$ in the following way:
$$
(r,n)\ll(s,m)\quad\text{if and only if}\quad\text{there is a $k \in \mathbb N$ with $nk=m$ and $s \ge rk$}.
$$
The relation $\ll$ is a partial order; reflexivity is obvious (take $k=1$) and transitivity is a simple calculation. Suppose that that $(r,n)\ll(s,m)$ and $(s,m)\ll(r,n)$; then there $k,k'$ with
$nk=m$, $mk'=n$, $s \ge rk$ and $r \ge sk'$. Then we get that $k = 1/k'$ but since $k$ is an integer the only possibility is $k=k'=1$; so $(r,n) = (s,m)$ and $\ll$ is antisymmetric.

\begin{lemma}\label{cases} \begin{enumerate}
\item $\alg K_{r,n} \le \alg K_{s,m}$ if and only if $(r,n)\ll (s,m)$;
\item $\alg K_{r,n} \le \alg K_{m,m}$ if and only  if $n\mathrel{|} m$;
\end{enumerate}
\end{lemma}
\begin{proof} 1. follows from two observations, both very easy to check. The first is that the subalgebras of $K(\alg \L_m) = \alg K_{m,m}$ are exactly the admissible subalgebras of $K(\alg \L_n)$ for all $n\mathrel{|} m$. The second
is that if $a$ is the unique coatom of $\alg \L_n$ (that hence generates $\alg \L_n$), then
$$
a^r \oplus a^s = \left\{
                   \begin{array}{ll}
                     a^{s+r-n}, & \hbox{if $s+r \ge n$;} \\
                     1, & \hbox{if $s+r< n$.}
                   \end{array}
                 \right.
$$
Clearly 2. is a straightforward consequence of 1.
\end{proof}

We consider a pair $(I,K)$ where $I \sse \mathbb N^\nabla$ is finite and $K$ either $\{0\}$ or $\emptyset$. From Theorem \ref{smallsub} it is evident that to any small subvariety of the first two forms we may associate a pair of this kind (the algebra associated to $\{0\}$ is $\alg K_\o$).

\newcommand{\dd}{\not\mathrel{||}}

\begin{theorem}\label{triples} Let $\vv V$ be  small subvariety of $\mathsf{WKL}$ with associated pair $(I,K)$; then
\begin{enumerate}
\item $\alg K_{r,n} \in \vv V$ if and only if either there is an $(s,m) \in I$ with $(r,n)\ll(s,m)$ or else there is an $m \in J$ with $n\mathrel{|}m$;
\item $\alg K_\o \in \vv V$ if and only if $K =\{0\}$;
\end{enumerate}
\end{theorem}
\begin{proof} The ``if" direction is obvious. Next we note that J\'onnson lemma implies that if $\alg B,\alg A_1,\dots,\alg A_n$ are subdirectly irreducible, then $\alg B \in \VV(\alg A_1,\dots,\alg A_n)$ if and only if $\alg B \in \VV(\alg A_i)$ for some $i\le n$.  For 1. it is obvious that $\alg K_{r,n} \notin \VV(\alg K_\o)$, hence either $\alg K_{r,n} \in \VV(\alg K_{s,m})$ or $\alg K_{r,n} \in \VV(\alg K^\o_m)$. The conclusion then follows from Lemma \ref{cases} (and again J\'onsson Lemma).

For 2.,  $\alg K^\o \notin \VV(\alg K_{r,n})$ for any $(r,n)$, since the latter satisfies $(x\meet 1)^n \app (x \meet 1)^{n+1}$ and the former does not.
\end{proof}

We call a pair $(I,K)$ {\bf reduced} if the following hold:
\begin{enumerate}
\ib $I \cup K \ne \emptyset$;
\ib if $(r,n) \in I$, then $(r,n) \not\ll (s,m)$ for all $(s,m) \in I\setminus \{(r,n)\}$;
\end{enumerate}

\begin{corollary} The small subvarieties of $\mathsf{WKL}$ of the first two forms according to Theorem \ref{smallsub} are in one-to-one correspondence with reduced apirs via the mapping
\begin{align*}
&(I, \emptyset) \longmapsto \VV(\{\alg K_{r,n}: (r,n) \in I\})\\
&(I,\{0\}) \longmapsto \VV(\{\alg K_{r,n}: (r,n) \in I\} \cup \{\alg K_\o\}).
\end{align*}
\end{corollary}

Now we are able to draw the lattice of subvarieties of any small subvariety of $\mathsf{WKL}$ of the first two forms. We also observe that
any small variety has the FEP if and only if it is locally finite, if and only if is generated by finite algebras (this is an easy consequence of the analogous statement for proper varieties of Wajsberg hoops, proved in \cite{Agliano2015}).  As far as the splittings are concerned we have:

\begin{lemma}\label{k4split} Let $\vv V$ be a proper variety of Wajsberg hoops not containing any $\alg \L^\o_n$ and suppose that $K(\vv V)$ has $(I,K)$ as a reduced pair. Then  $\alg K_4$ is splitting in $K(\vv V)$  and its conjugate variety is the subvariety of $K(\vv V)$, given by $\vv W=\tilde{\vv W}\cap K(\vv V)$, where the associate triple of $\tilde{\vv W}$ is $(\{(n_1-1,n_1),\dots,(n_k-1,n_k)\},K)$.
\end{lemma}
\begin{proof} First we observe that we can assume $I=\{(n_1,n_1),\dots,(n_k,n_k)\}$ with $(n_i,n_i) \ne (1,1)$ for $i=1,\dots,k$. Second, we note that, in general, $K_{n,n} \in \VV(K_{r,m})$ if and only if $n\mathrel{|} m$ and $r=m$. It follows that $\alg K_4 = \alg K_{1,1} \notin \VV(K_{n_{i-1},n_i})$ for $i\le k$; since $\alg K_4 \notin \VV(\alg K_\o)$ we conclude that $\alg K_4 \notin \vv W$. Conversely let $\vv U \sse K(\vv V)$, suppose that $\alg K_4 \notin U$ and let $(I'',K'')$ be the reduced pair associated to $\vv U$. Since $\vv U \sse \vv V$, $K''\sse K$. Let now $(r,n) \in I''$; then by Theorem \ref{triples}, $(r,n) \ll (n_i,n_i)$ for some $i$.  Now $r \ne n$, otherwise $\alg K_4 \in \vv U$; it follows that $(r,n) \ll (n_i-1,n_i)$. This means that for any $(r,n) \in I''$ there is an $i$ such that $(r,n) \ll (n_i-1,n_i)$; clearly this implies $\vv U \sse \vv W$ and so $\alg K_4$ is splitting in $K(\vv V)$.
\end{proof}

Since any variety is contained in a Kalman variety we get at once that $\alg K_4$ is splitting in any small subvariety of $\mathsf{WKL}$ to which it belongs.

Let's now draw some simple lattices; the first in Figure \ref{l2l3} is the lattice of subvarieties of $K(\VV(\alg \L_2,\alg \L_3))$. We have labeled only the subvarieties generated by a single algebra, since the others can be deduced using J\'onsson Lemma; $\vv W$ is the conjugate variety of $\alg K_4$.

\begin{figure}[htbp]
\begin{center}
\begin{tikzpicture}
\draw (0,2) -- (0,3) -- (0,4) -- (-1,5) -- (-2,6) -- (-3,7) -- (-3,6) ;
\draw (0,3) --(-1,4) -- (-2,5) -- (-3,6) -- (-3,7);
\draw (0,3) -- (1,4) -- (2,5) -- (2,6) -- (1,5) -- (0,4) -- (0,3);
\draw (2,5) -- (1,6) -- (0,7) -- (-1,8) -- (-1,9) -- (0,8) -- (1,7) -- (2,6) -- (2,5);
\draw (-1,8) -- (-2,7) -- (-3,6) --(-3,7) -- (-2,8) -- (-1,9);
\draw (-2,5) -- (-1,6) -- (0,7) -- (0,8) -- (-1,7) -- (-2,6) -- (-2,5);
\draw (-1,4) -- (0,5) -- (1,6) -- (1,7) -- (0,6) -- (-1,5) -- (-1,4);
\draw (1,4) -- (0,5) -- (-1,6) -- (-2,7) -- (-2,8) -- (-1,7) -- (0,6) -- (1,5) -- (1,4);
\draw[fill] (0,2) circle [radius=0.05];
\draw[fill] (0,3) circle [radius=0.05];
\draw[fill] (0,4) circle [radius=0.05];
\draw[fill] (0,5) circle [radius=0.05];
\draw[fill] (0,6) circle [radius=0.05];
\draw[fill] (0,7) circle [radius=0.05];
\draw[fill] (0,8) circle [radius=0.05];
\draw[fill] (-1,4) circle [radius=0.05];
\draw[fill] (-1,5) circle [radius=0.05];
\draw[fill] (-1,6) circle [radius=0.05];
\draw[fill] (-1,7) circle [radius=0.05];
\draw[fill] (-1,8) circle [radius=0.05];
\draw[fill] (-1,9) circle [radius=0.05];
\draw[fill] (-2,5) circle [radius=0.05];
\draw[fill] (-2,6) circle [radius=0.05];
\draw[fill] (-2,7) circle [radius=0.05];
\draw[fill] (-2,8) circle [radius=0.05];
\draw[fill] (-3,6) circle [radius=0.05];
\draw[fill] (-3,7) circle [radius=0.05];
\draw[fill] (1,4) circle [radius=0.05];
\draw[fill] (1,5) circle [radius=0.05];
\draw[fill] (1,6) circle [radius=0.05];
\draw[fill] (1,7) circle [radius=0.05];
\draw[fill] (2,5) circle [radius=0.05];
\draw[fill] (2,6) circle [radius=0.05];
\node[right] at (0,2) {\tiny $\mathsf{T}$};
\node[right] at (0,3) {\tiny $\VV(\alg K_3)$};
\node[right] at (0,4) {\tiny $\VV(\alg K_4)$};
\node[right] at (1,4) {\tiny $\VV(\alg K_{0,2})$};
\node[right] at (2,5) {\tiny $\VV(\alg K_{1,2})$};
\node[right] at (2,6) {\tiny $\VV(\alg K_{2,2})$};
\node[left] at (-1,4) {\tiny $\VV(\alg K_{0,3})$};
\node[left] at (-2,5) {\tiny $\VV(\alg K_{1,3})$};
\node[left] at (-3,6) {\tiny $\VV(\alg K_{2,3})$};
\node[left] at (-3,7) {\tiny $\VV(\alg K_{3,3})$};
\node[right] at (-1,8) {\tiny $\vv W$};
\node[right] at (-1,9) {\tiny $\VV(\alg K_{2,2}, \alg K_{3,3})$};
\node at (6,7) {\footnotesize $\vv W= \VV(\alg K_{1,2},\alg K_{2,3})$};
\end{tikzpicture}
\end{center}
\caption{$\Lambda(K(\VV(\alg \L_2,\alg \L_3)))$\label{l2l3}}
\end{figure}

The second lattice we are going to draw is the lattice of subvarieties of $\VV(K(\alg \L_3),\alg K_\o)$ (Figure \ref{l3Komega}); again $\vv W$ is the conjugate variety of $\alg K_4$.

\begin{figure}[htbp]
\begin{center}
\begin{tikzpicture}
\draw (0,2) -- (0,3) -- (0,4) -- (-1,5) -- (-2,6) -- (-3,7) -- (-3,6) ;
\draw (0,3) --(-1,4) -- (-2,5) -- (-3,6) -- (-3,7);
\draw (0,3) -- (1,4) -- (1,5) -- (0,4) -- (0,3);
\draw (-2,7) -- (-3,6) --(-3,7) -- (-2,8) ;
\draw (-2,5) -- (-1,6) --(-1,7) -- (-2,6) -- (-2,5);
\draw (-1,4) -- (0,5) -- (0,6) -- (-1,5) -- (-1,4);
\draw (1,4) -- (0,5) -- (-1,6) -- (-2,7) -- (-2,8) -- (-1,7) -- (0,6) -- (1,5) -- (1,4);
\draw[fill] (0,2) circle [radius=0.05];
\draw[fill] (0,3) circle [radius=0.05];
\draw[fill] (0,4) circle [radius=0.05];
\draw[fill] (0,5) circle [radius=0.05];
\draw[fill] (0,6) circle [radius=0.05];
\draw[fill] (-1,4) circle [radius=0.05];
\draw[fill] (-1,5) circle [radius=0.05];
\draw[fill] (-1,6) circle [radius=0.05];
\draw[fill] (-1,7) circle [radius=0.05];
\draw[fill] (-2,5) circle [radius=0.05];
\draw[fill] (-2,6) circle [radius=0.05];
\draw[fill] (-2,7) circle [radius=0.05];
\draw[fill] (-2,8) circle [radius=0.05];
\draw[fill] (-3,6) circle [radius=0.05];
\draw[fill] (-3,7) circle [radius=0.05];
\draw[fill] (1,4) circle [radius=0.05];
\draw[fill] (1,5) circle [radius=0.05];
\node[right] at (0,2) {\tiny $\mathsf{T}$};
\node[right] at (0,3) {\tiny $\VV(\alg K_3)$};
\node[right] at (0,4) {\tiny $\VV(\alg K_4)$};
\node[right] at (1,4) {\tiny $\VV(\alg K_\o)$};
\node[right] at (-2,8){\tiny $\VV(\alg K_{3,3},\alg K_\o)$};
\node[left] at (-1,4) {\tiny $\VV(\alg K_{0,3})$};
\node[left] at (-2,5) {\tiny $\VV(\alg K_{1,3})$};
\node[left] at (-3,6) {\tiny $\VV(\alg K_{2,3})$};
\node[left] at (-3,7) {\tiny $\VV(\alg K_{3,3})$};
\node[right] at (-2,7) {\tiny $\vv W$};
\node[right] at (3,7) {\footnotesize $\vv W= \VV(\alg K_{2,3},\alg K_\o)$};
\end{tikzpicture}
\end{center}
\caption{$\Lambda(K(\VV(\alg K_{3,3},\alg K_\o)))$\label{l3Komega}}
\end{figure}

\subsection{Basic Kalman lattices}

In this section we will examine the Kalman variety associated to basic hoops; the members of $K(\mathsf{BH})$ are called {\bf basic K-lattices} and we denote the variety by $\mathsf{KBH}$.
Basic hoops have been thoroughly investigated in the past twenty years,(the seminal paper is \cite{AFM}). The lattice $\Lambda (\mathsf{BH})$ although very complex, is still manageable and  the main reason is that a subdirectly irreducible basic hoop is the ordinal sum of Wajsberg hoops in an essentially unique way \cite {AglianoMontagna2003}. This fact has many interesting consequences; from now on when we write $\alg A = \bigoplus_{i \in I}\alg A_i$ for a totally ordered basic hoop, we always mean that the $\alg A_i$ are its Wajsberg components. The {\bf index} of $\alg A = \bigoplus_{i \in I}\alg A_i$ is $|I|$ if $|I|$ is finite and it is $\infty$ otherwise;  in general \cite{AglianoMontagna2003} $|I|\le n$ if and only if $\alg A$ satisfies
\begin{equation}
\bigwedge_{i=0}^{n-1}((x_{i+1} \imp x_i) \imp x_i) \le \bigvee_{i=0}^n x_i. \tag{$\lambda_n$}
\end{equation}
So if a basic chain $\alg A$ has finite index, the index of any other chain in $\VV(\alg A)$ is bounded by the index of $\alg A$. Chains of finite index can be classified according to the type of their Wajsberg components; this has been done for instance in \cite{AglianoMontagna2003},\cite{AglianoMontagna2017} and \cite{EstevaGodoMontagna2004}.  We will not explore fully the theory of chains of finite index in this paper, since our focus is on K-lattices. However it is evident that if $\alg A$ is a chain of finite index whose components are $\alg \L_m$  or $\alg \L^\o_m$ for some $m$, or else $\alg C_\o$ we can use the results of the previous sections to transfer information about $\Lambda(\VV(\alg A))$ to $\Lambda(\VV(K(\alg A)))$. Let's work out a couple of examples.

Let $\Omega(\vv C)$ be the variety of basic hoops satisfying  $x \imp x^2 \app x$.  In \cite{AglianoGalatosMarcos2020} the following facts have been proved:
\begin{enumerate}
\ib any subdirectly irreducible algebra in $\Omega(\vv C)$ is an ordinal sum of cancellative hoops;
\ib $\Omega(\vv C)$ is generated by any chain  of infinite index  and by all the chains of finite index;
\ib every proper subvariety of $\Omega(\vv C)$ is generated by a chain $\bigoplus_{n=0}^k \alg C_\o$ for some $k \in \mathbb N$; $\Lambda(\Omega(\vv C))$ is a chain of order $\o+1$.
\end{enumerate}

What about $K(\Omega(\vv C))$? Let's define $\alg K_\o^\o = \bigoplus_{n\in \mathbb N} \alg C_\o$ and $\alg K_\o^k= \bigoplus_{n=0}^k \alg C_\o$. It is as straightforward consequence of the theory developed so far that $K(\Omega(\vv C)) = \VV(K_\o^\o)$ and that $\VV(\alg K^k_\o)$ for $K \in \mathbb N$ are exactly the Kalman subvarieties  of $K(\Omega(\vv C))$. Moreover any
subdirectly irreducible in $\HH\SU\PP_u(\alg K_\o^k)$ is subalgebra of  $K(\bigoplus_{i=0}^k \alg A_i)$ for $\alg A_i \in \vv C$. Now the admissible subalgebras of  $K(\bigoplus_{i=0}^k \alg A_i)$ correspond (by Lemma \ref{admissibleordinalsum}) to the admissible subalgebras of $K(\alg A_0)$. But by Lemma \ref{tocancellative} the only admissible subalgebra of $K(\alg A_0)$ is itself, hence (again by Lemma \ref{admissibleordinalsum}), the only admissible subalgebra of $K(\bigoplus_{i=0}^k \alg A_i)$ is itself. The non-admissible subalgebras are just the admissible subalgebras of  $K(\alg B)$ for some subalgebra $\alg B$ of $\bigoplus_{i=0}^k \alg A_i$. But such subalgebras are all of the form $\bigoplus_{j=0}^l \alg B_j$ where $\{j_0,\dots, j_l\} \sse \{0,\dots, k\}$ and $\alg B_j$ is a subalgebra of some $\alg A_i$. In conclusion any member of $\HH\SU\PP_u(\alg K_\o^k)$, either generates $\VV(K^k_\o)$ or else it generates some $\VV(K^l_\o)$ with $l < k$. We have thus proved:

\begin{theorem} The lattice of subvarieties of $K(\Omega(\vv C))$ is a countable chain of order $\o+1$. The proper subvarieties are all of the form $\VV(\alg K^k_\o)$ for some $k \in \mathbb N$.
\end{theorem}

Since every finite chain has finite index, it makes sense to consider  the class of Kalman varieties associated to the varieties of basic hoops generated by a finite chain.
A finite totally ordered hoop is always of the form $\bigoplus_{i=0}^k \alg \L_{n_k}$; moreover
\begin{enumerate}
\ib any finite totally ordered chain  in $\VV(\bigoplus_{i=0}^k \alg \L_{n_k})$ is of the form $\bigoplus_{j=0}^l \alg \L_{m_j}$ where $l \le k$,and there are $k_0 < \dots <k_l$ such that
$m_j\mathrel{|} n_{k_j}$;
\ib if $s = \op{lcm}\{n_0,\dots, n_k\}$ then  each hoop in $\VV(\bigoplus_{i=0}^k \alg \L_{n_k})$ is $s$-potent, i.e. it satisfies $x^s \app x^{s+1}$;
\ib the lattice $\Lambda(\VV(\bigoplus_{i=0}^k \alg \L_{n_k}))$ is finite and each subvariety is finitely generated.
\end{enumerate}
Let's observe a simple fact:

\begin{lemma}\label{npotent} Let $\alg A \in \mathsf{CIRL}$ be $n$-potent for some $n$; then $K(\alg A)$ is $n+1$-potent.
\end{lemma}
\begin{proof} A standard induction argument shows that $(a,b)^{n+2} = (a^{n+2}, a^{n+1} \imp b)$ holds in $K(\alg A)$ for any $\alg A \in \mathsf{CIRL}$. Therefore if $\alg A \in \mathsf{CIRL}$ is $n$-potent, $a^n = a^{n+1}=a^{n+2}$ and
$$
(a,b)^{n+2} = (a^{n+1}, a^n \imp b) = (a,b)^{n+1}.
$$
\end{proof}

\begin{corollary} Let $\vv V$ be a subvariety of $\mathsf{KL}$; if $\vv V^-$ is $n$-potent (i.e. consists entirely of $n$-potent algebras) then  $\vv V$ is $n+1$-potent.
\end{corollary}
\begin{proof} Since $K(\vv V^-) =\{\alg B: \alg B \le K(\alg A), \alg A \in \vv V^-\}$, we get by Lemma \ref{npotent} that $K(\vv V^-)$ consists entirely of $n+1$-potent algebras.
The conclusion follows from $\vv V \sse K(\vv V^-)$.
\end{proof}

To draw the lattice of subvarieties of $K(\VV(\alg A))$ for a finite totally ordered hoop $\alg A$ we can use  Theorem \ref{admissibleordinalsum} and Lemma \ref{admissiblelemma}. These lattices have of course a high degree of symmetry, but become very complex
even for small values of $k, n_0, \dots,n_k$; in Figure \ref{l1plusl2} we have drawn the lattice of subvarieties of $K(\VV(\alg 2\oplus\alg \L_2))$ (here $\alg T$ is the only proper admissible subalgebra of $K(\alg 2\oplus \alg \L_2)$).
The reader can have fun in trying to draw the lattice of subvarieties of $K(\VV(\alg \L_2\oplus\alg \L_2))$ (warning: a large piece of paper might be handy!).

\begin{figure}[htbp]
\begin{center}
\begin{tikzpicture}[scale=.7]
\draw (0,0) -- (0,1) -- (0,2) -- (-1,2.5) -- (-2,3) -- (-1,3.5) -- (0,4) -- (0,5) -- (0,6) -- (0,7);
\draw (0,1) -- (1,1.5) -- (2,2) -- (2,3) -- (2,4) -- (1,4.5) -- (1,5.5) -- (0,6);
\draw (0,2) -- (1,2.5) -- (2,3) -- (1,3.5) -- (0,4);
\draw (-1,2.5) -- (0,3) -- (1,3.5) -- (1,4.5) -- (0,5);
\draw (1,1.5) -- (1,2.5) -- (0,3) -- (-1,3.5);
\draw[fill] (0,0) circle [radius=0.05];\node[right] at (0,0) {\tiny $\mathsf{T}$};
\draw[fill] (0,1) circle [radius=0.05];\node[left] at (0,1) {\tiny $\VV(\alg K_3)$};
\draw[fill] (0,2) circle [radius=0.05];\node[left] at (0,1.9) {\tiny $\VV(\alg K_4)$};
\draw[fill] (-1,2.5) circle [radius=0.05];\node[left] at (-1,2.4) {\tiny $\VV(\alg K_8)$};
\draw[fill] (-2,3) circle [radius=0.05];\node[left] at (-2,3) {\tiny $\VV(\alg K_9)$};
\draw[fill] (-1,3.5) circle [radius=0.05];
\draw[fill] (0,4) circle [radius=0.05];
\draw[fill] (0,5) circle [radius=0.05];
\draw[fill] (0,6) circle [radius=0.05];
\draw[fill] (0,7) circle [radius=0.05];\node[right] at (0,7) {\tiny $\VV(K(\alg 2 \oplus \alg \L_2))$};
\draw[fill] (1,1.5) circle [radius=0.05];\node[right] at (1,1.4) {\tiny $\VV(\alg K_{0,2})$};
\draw[fill] (2,2) circle [radius=0.05];\node[right] at (2,2) {\tiny $\VV(\alg K_{1,2})$};
\draw[fill] (2,3) circle [radius=0.05];
\draw[fill] (2,4) circle [radius=0.05];\node[right] at (2,4) {\tiny $\VV(\alg K_{2,2})$};
\draw[fill] (1,4.5) circle [radius=0.05];
\draw[fill] (1,5.5) circle [radius=0.05];\node[right] at (1,5.5) {\tiny $\VV(\alg T)$};
\draw[fill] (0,3) circle [radius=0.05];
\draw[fill] (1,2.5) circle [radius=0.05];
\draw[fill] (1,3.5) circle [radius=0.05];
\end{tikzpicture}
\caption{$\Lambda(\VV(K(\alg 2\oplus \alg \L_2)))$\label{l1plusl2}.}
\end{center}
\end{figure}
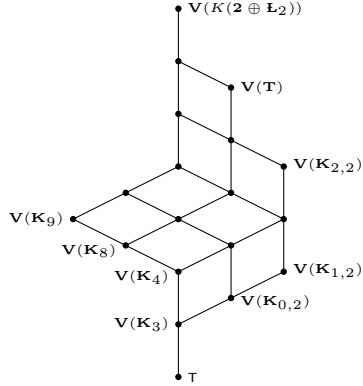

\section*{Conclusions and future work}

We have studied the variety $\mathsf{KL}$ of Kalman lattices, especially the lower part of the lattice of subvarieties $\lambda(\mathsf{KL})$. This is interesting by itself, and the machinery we have developed could be used to investigate other interesting (Kalman) subvarieties of $\mathsf{KL}$. Additionally, it might shed more light on $\Lambda(\mathsf{CRL})$, a rather unknown object.

Even though we have investigated the lattice of various subvarieties of K-lattices to a certain extent, as always there are some interesting unanswered questions.

One question posed in Section \ref{section3} is whether there are more atoms in $\Lambda(\mathsf{KL})$ other than $\mathsf{PKL}$ and $K(\mathsf{C})$. This problem is directly related to finding atoms in $\Lambda(\mathsf{CIRL})$ different from $\mathsf{GBA}$ and $\mathsf{C}$.

Another path to explore would be to characterize the projective objects in subvarieties of $\mathsf{KL}$; an algebra $\alg A$ is {\bf projective} in a variety $\vv V$ of similar algebras, if it is a retract of some free algebra in $\vv V$. It can be shown that every subdirectly irreducible projective algebra is splitting, so in general the class of projective algebras is larger than the class of splitting algebras. This might make more transparent the relation between the projective algebras in $\vv V \sse \mathsf{CIRL}$ and the projective algebras in $K(\vv V)$.

However the most natural path leads to the {\em bounded} case. If we enlarge the type of $\mathsf{CIRL}$ with a constant $0$ and we add the axiom $0 \le x$ then we get the variety of {\bf bounded} commutative residuated lattices, denote by $\mathsf{BCRL}$. Any  finite algebra in $\mathsf{CIRL}$ is naturally bounded and most subvarieties of $\mathsf{CIRL}$ have their bounded version. As a matter of fact the connection is strong, but there are also significant differences that are reflected also by the {\em kalmanization} (defined in the natural way) of those varieties. We intend to explore this circle of ideas in a subsequent paper.

\providecommand{\bysame}{\leavevmode\hbox to3em{\hrulefill}\thinspace}
\providecommand{\MR}{\relax\ifhmode\unskip\space\fi MR }
\providecommand{\MRhref}[2]{%
  \href{http://www.ams.org/mathscinet-getitem?mr=#1}{#2}
}
\providecommand{\href}[2]{#2}

\end{document}